\documentclass[11pt]{article}
\usepackage{color}
\usepackage{graphicx}
\usepackage{amsmath,amsfonts,amsthm}
\usepackage{amssymb,latexsym}
\usepackage{fixmath}
\usepackage{mathrsfs,amsbsy}
\usepackage{dsfont}
\usepackage{enumerate}
\usepackage{algorithm,algorithmic}
\usepackage{kbordermatrix}
\usepackage{xcolor}
\usepackage{soul,xcolor}
\setstcolor{red}
\def\bbC{\mathbb{C}}
\def\bbO{\mathbb{O}}
\def\bbR{\mathbb{R}}
\def\bbX{\mathbb{X}}

\def\cR{{\cal R}}
\def\cU{{\cal U}}
\def\cV{{\cal V}}
\def\cX{{\cal X}}

\def\sss{\scriptscriptstyle}

\DeclareMathOperator{\diag}{diag}
\DeclareMathOperator{\dist}{dist}

\DeclareMathOperator*{\opt}{opt}
\DeclareMathOperator{\ui}{ui}

\DeclareMathOperator{\rank}{rank}
\DeclareMathOperator{\re}{Re}
\DeclareMathOperator{\tr}{tr}

\DeclareMathOperator{\F}{F}

\DeclareMathOperator{\T}{T}
\DeclareMathOperator{\UI}{ui}

\def\wtd{\widetilde}
\def\what{\widehat}

\def\hadm{{\tt hadamard}}
\def\randn{{\tt randn}}
\def\orth{{\tt orth}}
\def\bestr{\mbox{\scriptsize\rm best-$r$}}

\newtheorem{theorem}{Theorem}[section]
\newtheorem{lemma}{Lemma}[section]

\theoremstyle{definition}

\newtheorem{remark}{Remark}[section]
\newtheorem{example}{Example}[section]

\numberwithin{equation}{section}
\numberwithin{figure}{section}
\numberwithin{table}{section}

\title{Variations of Orthonormal Basis Matrices of Subspaces}

\author{Zhongming Teng\thanks{%
   College of Computer and Information Science, Fujian Agriculture and Forestry University, Fuzhou, 350002, P. R. China.
Email: {\tt zhmteng@fafu.edu.cn}}
\and  Ren-Cang Li\thanks{%
   Department of Mathematics, University of Texas at Arlington, Arlington, TX 76019-0408,
USA.  E-mail: {\tt rcli@uta.edu}.  Supported in part by NSF DMS-2009689.}
}
\date{\today}

\begin{document}
\maketitle

\begin{abstract}
An orthonormal basis matrix $X$ of a subspace ${\cal X}$ is known not to be unique, unless there are some kinds of normalization requirements.
One of them is to require that $X^{\T}D$ is positive semi-definite, where $D$ is a
constant matrix of apt size. It is  a natural one in multi-view subspace learning
models in which $X$ serves as a projection matrix and is determined by
a maximization problem over the Stiefel manifold whose objective function contains
and increases with $\tr(X^{\T}D)$.
This paper is concerned with bounding the change in orthonormal basis matrix $X$ as subspace ${\cal X}$ varies
under the requirement that $X^{\T}D$ stays positive semi-definite.
The results are useful in
convergence analysis of the NEPv approach (nonlinear eigenvalue problem
with eigenvector dependency) to solve the maximization problem.
\end{abstract}

\medskip
{\small
{\bf Key words.} Unitarily invariant norm, canonical angle, subspace,
orthonormal basis matrix,  optimization on Stiefel manifold.

\medskip
{\bf Mathematics subject classifications (2010).} 15A45, 65F35
}

\section{Introduction}\label{sec:intro}
Recently in \cite{luli:2022}, the following optimization problem over the Stiefel manifold $\bbO^{n\times k}$
\begin{equation}\label{eq:OptSTM}
\max_{X\in\bbO^{n\times k}} f(X)\quad\mbox{with}\,\, f(X):=\phi(X)+\psi(X)\times\tr(X^{\T}D)
\end{equation}
is  considered, where $D$ is a constant matrix, the Stiefel manifold
$$
\bbO^{n\times k}=\{X\in\bbR^{n\times k}\,:\, X^{\T}X=I_k\,\,(\mbox{the $k\times k$ identity matrix})\},
$$
$\phi$ and $\psi$ are two unitarily invariant functions in the sense
that $\phi(XQ)\equiv\phi(X)$ and $\psi(XQ)\equiv\psi(X)$ for any $Q\in\bbO^{k\times k}$ and $\psi(X)>0$.
It is an abstraction of a few concrete problems arising from subspace learning
\cite{wazl:2022a,wazl:2022,zhys:2020,zhwb:2022},
where objective functions  contain $\tr(X^{\T}D)$ and
increase with $\tr(X^{\T}D)$.
Along the lines of earlier research in \cite{wazl:2022,zhys:2020,zhwb:2022},
the authors of \cite{luli:2022} started by transforming the KKT condition for \eqref{eq:OptSTM}
into an NEPv (nonlinear eigenvalue problem with eigenvector dependency)
\begin{equation}\label{eq:NEPv-form}
H(X)X=X\Omega,\quad X\in\bbO^{n\times k},
\end{equation}
where $H(X)$, dependent of $X$, is symmetric. NEPv of this form are not new, however, and in fact before
\cite{zhln:2010,zhln:2013} where orthogonal linear discriminant analysis (OLDA)
was first solved through NEPv, they mostly come from solving discretized
Kohn-Sham equations from the density functional theory \cite{hoko:64,kosh:1965,robl:2012,yaml:2009}.
Numerically, NEPv~\eqref{eq:NEPv-form} is often solved by the so-called {\em self-consistent-field\/} (SCF) iteration.
In \cite{luli:2022}, NEPv~\eqref{eq:NEPv-form} is solved by an SCF-type iteration:
given $X_0\in\bbO^{n\times k}$,
\begin{equation}\label{eq:SCF1}
\mbox{iteratively solve $H(X_{i-1})\what X_i=\what X_i\Omega_i$ for $\what X_i$ which is postprocessed to get $X_i$},
\end{equation}
until convergence, where the postprocessing yields $X_i=\what X_iQ_i$ for some $Q_i\in\bbO^{k\times k}$ such that
$X_i^{\T}D\succeq 0$ (positive semidefinite). $Q_i$ is often taken to be an orthogonal polar factor of
$X_i^{\T}D$ \cite{luli:2022,wazl:2022,zhys:2020,zhwb:2022},
owing to the fact that $f(X)$  is monotonically increasing
in $\tr(X^{\T}D)$ and that $Q_*\in\bbO^{k\times k}$ such that
$(XQ_*)^{\T}D=Q_*^{\T}(X^{\T}D)\succeq 0$ (positive semidefinite) ensures \cite{luli:2022,wazl:2022}
\begin{equation}\label{eq:tr-incr}
\tr([XQ_*]^{\T}D)=\tr(Q_*^{\T}[X^{\T}D])=\max_{Q\in\bbO^{k\times k}}\tr(Q^{\T}[X^{\T}D])\ge\tr( X^{\T}D),
\end{equation}
and the inequality is strict if $X^{\T}D\not\succeq 0$.

The SCF-type iteration \eqref{eq:SCF1} differs from
the classical SCF  for solving discretized
Kohn-Sham equations in its postprocessing from $\what X_i$ to $X_i$, which is not needed in the classical SCF for NEPv
that is unitarily invariant, i.e., $H(XQ)\equiv H(X)$ for any $Q\in\bbO^{k\times k}$. Before \cite{luli:2022},
SCF-type \eqref{eq:SCF1} had appeared in \cite{wazl:2022,zhys:2020,zhwb:2022}. Often indiscriminately, we use SCF to
refer to both the classical SCF and SCF-type iteration when no confusion arises.

An immediate consequence of \eqref{eq:tr-incr} is that $X_*^{\T}D\succeq 0$ for  any maximizer $X_*$
of maximization problem~\eqref{eq:OptSTM}. Another important characterization of maximizer $X_*$
is that \cite[Theorem 3.1]{luli:2022}
\begin{equation}\label{cond:rank}
\rank(X_*^{\T}D)=\rank(D).
\end{equation}
As a result, for any $X\in\bbO^{n\times k}$ such that the column space of $X$, denoted by  $\cR(X)$, is sufficiently close to $\cR(X_*)$, we have
$\rank(X^{\T}D)=\rank(D)$ \cite[Lemma 5.1]{luli:2022} which implies the continuity
of the canonical orthogonal polar factor
of $X^{\T}D$ for $\cR(X)$ near $\cR(X_*)$ \cite{high:2008,li:1993b}.

One of the key issues for SCF-type iteration \eqref{eq:SCF1}, as an iterative scheme, is whether
the generated sequence of approximations converge to the intended target. In the case when
optimization problem~\eqref{eq:OptSTM} is involved, that target is one of its maximizers. Because of technical
limitation, existing results on convergence are really about convergence-in-subspace, i.e., the convergence of
$\cR(X_i)$ to some $k$-dimensional subspace $\cX_*:=\cR(X_*)$ with exact $X_*$ unknown of course.
In other words, existing results may guarantee that $\cR(X_i)$ converges to
$\cX_*$ and produce estimates on the distance between subspaces $\cR(X_i)$ and
$\cX_*$ at convergence, but do not yield bounds on $\|X_i-X_*\|$ where $\|\cdot\|$ is some matrix norm.
In the case of OLDA or any objective function $f$ that is
unitarily invariant, this is the best  we can do because if $X_*$ is an optimizer then so is $X_*Q$ for any
$Q\in\bbO^{k\times k}$, but for $f$ as in \eqref{eq:OptSTM}, the optimizer $X_*$
is provably unique, provided $\rank(X_*^{\T}D)=k$,  within the orbit
\begin{equation}\label{eq:orbit4X*}
\bbX_*:=\{X_*Q\,:\, Q\in\bbO^{k\times k}\}
\end{equation}
whose elements share the same subspace $\cX_*=\cR(X_*)$, but $X_*$ is only
partially unique when $\rank(X_*^{\T}D)<k$ \cite{wazl:2022,luli:2022}. Notice that
$\rank(X^{\T}D)$ is a constant for all $X\in\bbX_*$, independent of any particular orthonormal basis matrix for $\cX_*$.
In view of this discussion, we may regard $D$ as some kind of decider that picks up particular $X_*$ from the orbit
\eqref{eq:orbit4X*}.

The goal of this paper is to answer the following mathematical question:
\begin{equation}\label{eq:TheQuestion}
\framebox{
\parbox{12.7cm}{Given two $k$-dimensional subspaces $\cX$ and $\wtd\cX$ of $\bbR^n$,
let $X,\,\wtd X\in\bbO^{n\times k}$ be their orthonormal basis matrices, respectively,
such that $X^{\T}D,\,\wtd X^{\T}D\succeq 0$, and assume $\rank(X^{\T}D)=\rank(\wtd X^{\T}D)$.
How do we bound the difference between $X$ and $\wtd X$ in terms of
the difference between the subspaces $\cX$ and $\wtd\cX$?}
}
\end{equation}
As to the issue raised moments ago for the convergence of $X_i$ in SCF, our main result can be
used to
bound $\|X_i-X_*\|$ in terms of the distance between the subspaces $\cR(X_i)$ and
$\cX_*$, by letting $X=X_*$ and $\wtd X=X_i$. The first condition $X_i^{\T}D\succeq 0$ holds by design and
$X_*^{\T}D\succeq 0$ is a necessary condition for a maximizer, and
$\rank(X_i^{\T}D)=\rank(X_*^{\T}D)$ near convergence is due to \eqref{cond:rank} of
\cite[Lemma 5.1]{luli:2022}

{\bf Notation.}
$\bbR^{m\times n}$ is the set of $m\times n$ real matrices, $\bbR^n=\bbR^{n\times 1}$ and $\bbR=\bbR^1$. $I_n\in\bbR^{n\times n}$ is
the identity matrix.
For $B\in\bbR^{m\times n}$,
$\cR(B)$ is the column subspace, spanned by its columns, and its singular values are denoted by
$\sigma_i(B)$ for $i=1,\ldots,\min(m,n)$ arranged in the nonincreasing order,
and
$$
\|B\|_2=\sigma_1(B),\,\,
\|B\|_{\F}=\sqrt{\sum_{i=1}^{{\rm rank}(B)}[\sigma_i(B)]^2},\,\,
\|B\|_{\tr}=\sum_{i=1}^{{\rm rank}(B)}\sigma_i(B)
$$
are the spectral, Frobenius, and trace norms of $B$, respectively.
$B^{\T}$ is the transpose of $B$. The trace norm is also known as the nuclear norm.
For a symmetric $A\in\bbR^{n\times n}$,
$A\succ 0 ~(\succeq 0)$ means that $A$ is positive definite (semi-definite).
MATLAB-like notation is used to access the entries of a matrix or vector:
$X_{(i:j,k:l)}$ denotes the submatrix of a matrix $X$, consisting of the intersections of
rows $i$ to $j$ and columns $k$ to $l$, and when $i : j$ is replaced by $:$, it means all rows.

\section{Preliminaries}
In this section, we collect a few known results that we will need in our later developments.

\subsection{Canonical angles between subspaces}
Given two $k$-dimensional subspaces $\cX$ and $\wtd\cX$ of $\bbR^{n}$,
let $X\in\bbO^{n\times k}$ and $\wtd X\in\bbO^{n\times k}$ be their orthonormal basis matrices,
respectively, i.e.,
\[
X^{\T}X=I_k,\ \cR(X)=\cX\quad\mbox{and}\quad \wtd X^{\T}\wtd X=I_k,\ \cR(\wtd X)=\wtd \cX.
\]
Denote by $\omega_i$ for $1\le i\le k$ the singular values of $X^{\T}\wtd X$ in descending order, i.e.,
$\omega_1\ge\dots\ge\omega_k$. The $k$ canonical angles $\theta_i(\cX,\wtd\cX)$ between $\cX$ and $\wtd\cX$
are defined as
\[
0\le\theta_i(\cX,\wtd\cX):=\arccos(\omega_{k-i+1})\le\frac{\pi}{2},\quad\mbox{for\ $1\le i\le k$} .
\]
Set
\begin{equation}\label{eq:angle-def}
\Theta(\cX,\wtd\cX)=\diag\left(\theta_1(\cX,\wtd\cX),\dots,\theta_k(\cX,\wtd\cX)\right)\in\bbR^{k\times k}.
\end{equation}
It can be seen that the angle matrix $\Theta(\cX,\wtd\cX)$ in~\eqref{eq:angle-def} is independent of
choosing orthonormal basis matrices of $\cX$ and $\wtd\cX$.

In this paper, any unitarily invariant norm $\|\cdot\|_{\UI}$ \cite{stsu:1990} we refer to is assumed to be dimension-free in the sense
that it can be applied to matrices of any size consistently such as the matrix spectral and Frobenius norm.
Less stringently, we may limit our unitarily invariant norms that are induced by
a symmetric gauge function $\Phi$ on $\bbR^N$ \cite[section~II.4]{stsu:1990} with sufficiently large $N$ such that
all matrices $B$ of interest have no more than $N$ rows and columns, and then
we let \cite[p.79]{stsu:1990}
$$
\|B\|_{\UI}=\Phi(\sigma_1(B),\ldots,\sigma_r(B), 0,\ldots,0),
$$
i.e., appending $0$ to the set of singular values of $B$ to make $N$ of them.
It is known that for matrices $A$, $B$ and $C$ of compatible size
we have
\begin{equation}\label{eq:UI-2:consistent}
\|ABC\|_{\UI}\le\|A\|_2\|B\|_{\UI}\|C\|_2.
\end{equation}

Sun~\cite[p.95]{sun:1987} proved that for
any unitarily invariant norm $\|\cdot\|_{\UI}$,
$\|\sin\Theta(\cX,{\cal Y})\|_{\UI}$
defines a unitarily invariant metric on the Grassmann manifold
consisting of all $k$-dimensional subspaces
of $\bbR^n$. A convenient way to work with $\|\sin\Theta(\cX,{\cal Y})\|_{\UI}$ is as follows.
Let $X_{\perp},\,\wtd X_{\perp}\in\bbO^{n\times (n-k)}$
such that $[X, X_{\perp}]\in\bbO^{n\times n}$ and $[\wtd X, \wtd X_{\perp}]\in\bbO^{n\times n}$, respectively. Then
\begin{equation}\label{eq:def-comp-sin}
\|\sin\Theta(\cX,\wtd\cX)\|_{\UI}=\|X_{\perp}^{\T}\wtd X\|_{\UI}
=\|\wtd X_{\perp}^{\T}X\|_{\UI}.
\end{equation}

\begin{lemma}[{\cite[Lemma~4.1]{zhli:2014b}}]\label{le:max-angle}
There exists an orthogonal matrix $Q\in\bbO^{k\times k}$ such that
\begin{equation}\label{eq:angle1}
    \|\sin\Theta(\cX,\wtd\cX)\|_{\UI}\le\|X-\wtd XQ\|_{\UI}\le\sqrt 2\,\|\sin\Theta(\cX,\wtd\cX)\|_{\UI}.
\end{equation}
\end{lemma}

\subsection{SVD Perturbation}
For any matrix $B$ of apt size, we will use $B_{\bestr}$ to denote its {\em best rank-$r$ approximation\/}
obtained by zeroing out all of its singular values except the first $r$ largest ones in its SVD. It can be shown,
using Fan's theorem \cite[p.86]{stsu:1990}, that
for any unitarily invariant norm $\|\,\cdot\,\|_{\UI}$, $\|(\,\cdot\,)_{\bestr}\|_{\UI}$ for $r\ge 1$ is also
a unitarily invariant norm. The consistency inequalities in \eqref{eq:UI-2:consistent} can be sharpened a little:
\begin{equation}\label{eq:UI-2:consistent'}
\|ABC\|_{\UI}\le\| A\|_2\|B_{\bestr}\|_{\UI}\|C\|_2\quad\mbox{if $\min\{\rank(A),\rank(B),\rank(C)\}\le r$}.
\end{equation}
The next lemma is a corollary of the classical Wedin's result of
\cite[(3.1)]{wedi:1972} and \eqref{eq:UI-2:consistent'}.

\begin{lemma}\label{lm:SVD-pert:wedin}
Let $B,\,\wtd B=B+F\in\bbR^{m\times n}$ such that $\rank(B)=\rank(\wtd B)=r$ and let
their singular value decompositions be
\begin{equation}\label{eq:Bsvd}
    B=U\Sigma V^{\T}\quad\mbox{and}\quad
\wtd B=\wtd U\wtd\Sigma \wtd V^{\T},
\end{equation}
where $\Sigma_{(1:r,1:r)}\succ 0$ and $\wtd\Sigma_{(1:r,1:r)}\succ 0$.
Then we have
\begin{subequations}\label{eq:sin}
\begin{align}
\max\left\{\|\sin\Theta(\cU,\wtd\cU)\|_{\UI},\,\|\sin\Theta(\cV,\wtd\cV)\|_{\UI}\right\}
                &\le\frac{\|F_{\bestr}\|_{\UI}}{\max\{\sigma_r(B),\sigma_r(\wtd B)\}} \label{eq:sin-a} \\
                &\le\frac{\|F\|_{\UI}}{\max\{\sigma_r(B),\sigma_r(\wtd B)\}}. \label{eq:sin-b}
\end{align}
\end{subequations}
where $\cU=\cR(U_{(:,1:r)})$,
$\wtd\cU=\cR(\wtd U_{(:,1:r)})$,
$\cV=\cR(V_{(:,1:r)})$, and
$\wtd\cV=\cR(\wtd V_{(:,1:r)})$.
\end{lemma}

\begin{proof}
Let
\begin{align*}
R&:=B\wtd V_{(:,1:r)}-\wtd U_{(:,1:r)}\wtd\Sigma_{(1:r,1:r)}=(B-\wtd B)\wtd V_{(:,1:r)},\\
S&:=B^{\T}\wtd U_{(:,1:r)}-\wtd V_{(:,1:r)}\wtd\Sigma_{(1:r,1:r)}=(B-\wtd B)^{\T}\wtd U_{(:,1:r)}.
\end{align*}
By \eqref{eq:UI-2:consistent'}, we get
$\|R\|_{\UI}\le\|F_{\bestr}\|_{\UI}$ and $\|S\|_{\UI}\le\|F_{\bestr}\|_{\UI}$.
Hence, with the help of the classical Wedin's result of \cite[(3.1)]{wedi:1972} (see also \cite[Fact 4, p.21-7]{li:2014HLA})
for the case, we have
\begin{align}
\max\left\{\|\sin\Theta(\cU,\wtd\cU)\|_{\UI},\,\|\sin\Theta(\cV,\wtd\cV)\|_{\UI}\right\}
  &\le\frac{\max\{\|R\|_{\UI},\|S\|_{\UI}\}}{\sigma_r(\wtd B)} \nonumber \\
  &\le\frac{\|F_{\bestr}\|_{\UI}}{\sigma_r(\wtd B)}. \label{eq:wedin}
\end{align}
Switching the roles of $B$ and $\wtd B$ in \eqref{eq:wedin}, we get
\begin{equation}\label{eq:wedin'}
\max\left\{\|\sin\Theta(\cU,\wtd\cU)\|_{\UI},\,\|\sin\Theta(\cV,\wtd\cV)\|_{\UI}\right\}
\le\frac{\|F_{\bestr}\|_{\UI}}{\sigma_r(B)}.
\end{equation}
Inequalities in \eqref{eq:sin} are the consequences of \eqref{eq:wedin} and \eqref{eq:wedin'}.
\end{proof}

\subsection{Polar decomposition}
Any $B\in\bbR^{n\times m}$ $(n\ge m)$ can be decomposed as $B=QH$, called a polar decomposition \cite[p.449]{hojo:2013},
where $Q\in\bbO^{n\times m}$ and $H=(B^{\T}B)^{1/2}\succeq 0$ is the unique positive semidefinite square root of $B^{\T}B$.
It is known that orthogonal factor $Q$  is unique if and only if $\rank(B)=m$ \cite{li:2014HLA}. When $\rank(B)<m$,
there is  the so-called {\em canonical polar decomposition\/} $B=QH$ in which $Q\in\bbR^{n\times m}$
is a partial isometry and satisfies $\cR(Q^{\T})=\cR(H)$ and again $H=(B^{\T}B)^{1/2}$.
In the canonical polar decomposition, $Q$ is unique
(see \cite[p.220]{begr:2003}, \cite[chapter~8]{high:2008}, \cite{li:1993b}).


\begin{lemma}[{\cite[Theorem~1]{li:1995}}, {\cite[Theorem~3.4]{lisu:2003}}] \label{le:polar}
Suppose that $B\in\bbR^{n\times m}$  ($n\ge m$) is perturbed to $\wtd B$
such that $\rank(\wtd B)=\rank(B)=r$. Let the SVD of $B$ be given by
\begin{equation}\label{eq:svd-B}
B=\kbordermatrix{ &\sss r &\sss n-r\\
                                 & U_1 & U_2}\times
                      \kbordermatrix{ &\sss r &\sss m-r \\
                                    \sss r & \Sigma_1 & \\
                                    \sss n-r & & 0}\times
                      \kbordermatrix{ &\\
                               \sss r  & V_1^{\T} \\
                               \sss m-r  & V_2^{\T}},
\end{equation}
where $r=\rank(B)$, and similarly
the SVD of $\wtd B$ takes the form as in \eqref{eq:svd-B} except with
a tilde on each of the symbols there. Then
\begin{equation}\label{eq:QtQ}
 Q= U_1 V_1^{\T}
\quad\mbox{and}\quad
\wtd Q=\wtd U_1\wtd V_1^{\T}
\end{equation}
are
the unique partial isometry factors of the canonical polar decompositions
of $B$ and $\wtd B$, respectively, and
\begin{equation}\label{eq:polar-iu}
\|Q-\wtd Q\|_{\UI}\le
\begin{cases}
    \dfrac{2}{\sigma_n(B)+\sigma_n(\wtd B)}\,\|\wtd B-B\|_{\UI}, &\mbox{if $r=n=m$}; \\[1em]
    \left(\dfrac{2}{\sigma_r(B)+\sigma_r(\wtd B)}+\dfrac{2}{\max\{\sigma_r(B),\sigma_r(\wtd B)\}}\right)\,
               \|\wtd B-B\|_{\UI},
    &\mbox{otherwise}.
\end{cases}
\end{equation}
This inequality can be improved for the matrix spectral and Frobenius norm
in the case when $n>m$ or $r<n$:
\begin{subequations}\label{eq:polar-fnorm-2norm}
\begin{align}
\|Q-\wtd Q\|_{\F}
           &\le\dfrac{2}{\sigma_r(B)+\sigma_r(\wtd B)}\,\|\wtd B-B\|_{\F}, \label{eq:polar-fnorm} \\
\|Q-\wtd Q\|_2
          &\le \sqrt{\dfrac{4}{\left[\sigma_r(B)+\sigma_r(\wtd B)\right]^2}
                     +\dfrac{2}{\left[\max\{\sigma_r(B),\sigma_r(\wtd B)\}\right]^2}}\,\|\wtd B-B\|_2. \label{eq:polar-2norm}
\end{align}
\end{subequations}
\end{lemma}

The next lemma characterizes $X\in\bbO^{n\times k}$ such that $X^{\T}D\succeq 0$ into a sum of two terms, one of which depends on $\cR(X)$ only.

\begin{lemma}[{\cite[Theorem~3.2]{wazl:2022}}] \label{le:mp}
Given a $k$-dimensional subspace $\cX$ of $\bbR^n$, let $X_{\diamond}\in\bbO^{n\times k}$ with $\cR(X_{\diamond})=\cX$,
and let $r=\rank(X_{\diamond}^{\T}D)$ where $D\in\bbR^{n\times k}$. Let the SVD of $X_{\diamond}^{\T}D$ be
\begin{equation}\label{eq:svd-xdd}
X_{\diamond}^{\T}D=\kbordermatrix{ &\sss r &\sss k-r\\
                                 & U_1 & U_2}\times
                      \kbordermatrix{ &\sss r &\sss k-r \\
                                    \sss r & \Sigma_1 & \\
                                    \sss k-r & & 0}\times
                      \kbordermatrix{ &\\
                               \sss r  & V_1^{\T} \\
                               \sss k-r  & V_2^{\T}}.
\end{equation}
Any $X\in\bbO^{n\times k}$ with $\cR(X)=\cX$ such that $X^{\T}D\succeq 0$ takes the form
\[
X=X_{\diamond}U_1V_1^{\T}+X_{\diamond}U_2WV_2^{\T},
\]
where the first term, although constructed from $X_{\diamond}$, depends only on $\cX$, while
the second term has a freedom in $W\in\bbO^{(k-r)\times (k-r)}$. In particular, if $r=k$, then
$X=X_{\diamond}UV^{\T}$ is unique, given $\cX$.
\end{lemma}

\section{Main result}\label{sec:main}
In this section, we will present our main result that answers the question in \eqref{eq:TheQuestion}.
Given $k$-dimensional subspace $\cX$ of $\bbR^n$, consider
all $X\in\bbO^{n\times k}$ such that
\begin{equation}\label{eq:charz-X}
\cR(X)=\cX, \quad X^{\T}D\succeq 0.
\end{equation}
The first condition $\cR(X)=\cX$ merely says that $X$
is an orthonormal basis matrix of $\cX$ and it is not unique, and in fact, it has the degree of freedom:
$k^2-\frac 12 k(k+1)=\frac 12 k(k-1)$. It is the  second characterization $X^{\T}D\succeq 0$ that will decide
which one or ones among the orthonormal basis matrices of $\cX$ should  be.
Let $r=\rank(X^{\T}D)$ and $r'=k-r$, and let the SVD of $X^{\T}D$ be
\begin{equation}\label{eq:SVD:XD}
X^{\T}D=V\Sigma V^{\T}\equiv\kbordermatrix{ &\sss r &\sss r'\\
                                 & V_1 & V_2}\times
                      \kbordermatrix{ &\sss r &\sss r' \\
                                    \sss r & \Sigma_1 & \\
                                    \sss r' & & 0}\times
                      \kbordermatrix{ &\\
                               \sss r  & V_1^{\T} \\
                               \sss r'  & V_2^{\T}}.
\end{equation}
By Lemma~\ref{le:mp}, we know that  $Y\in\bbO^{n\times k}$
such that $\cR(Y)=\cX$ and $Y^{\T}D\succeq 0$ if and only if
\begin{equation}\label{eq:xset}
Y\in\bbX:=\left\{XV\begin{bmatrix}
                     I_r &  \\
                      & W
                   \end{bmatrix}V=
XV_1V_1^{\T}+XV_2WV_2^{\T}\, :\,W\in\bbO^{r'\times r'}\right\}.
\end{equation}
When $\rank(X^{\T}D)=k$, the second term $XV_2WV_2^{\T}$ disappears and $\bbX=\{X\}$, which means that $X$ is unique.
But when $\rank(X^{\T}D)<k$, $\bbX$ is parameterized by a matrix variable $W\in\bbO^{r'\times r'}$ that
has the degree of freedom $\frac 12r'(r'-1)$. Since any element in $\bbX$ could be taken as $X$ to begin with, $X$ is not
uniquely decided by $X^{\T}D\succeq 0$. It worths emphasizing that $XV_1V_1^{\T}$ in \eqref{eq:xset}
depends on $\cX$ only, although it is constructed with the help of a particular orthonormal basis matrix $X$ of
$\cX$.

The same can be said about $\wtd X\in\bbO^{n\times k}$ such that
\begin{equation}\label{eq:charz-tX}
\cR(\wtd X)=\wtd\cX, \quad \wtd X^{\T}D\succeq 0.
\end{equation}
In view of those, it only makes sense to bound $\|X-\wtd X\|$ when $\rank(X^{\T}D)=\rank(\wtd X^{\T}D)=k$ but to bound
$\min \|\wtd X-Y\|$ subject to $Y\in\bbX$ when $\rank(X^{\T}D)=\rank(\wtd X^{\T}D)<k$.

Let $r=\rank(X^{\T}D)=\rank(\wtd X^{\T}D)$, and express the SVD of
$\wtd X^{\T}D$ in the same way as in \eqref{eq:SVD:XD}, except putting a {\em tilde\/} on
each of the symbols $V_i$ and $\Sigma_1$ there. Write
\begin{equation}\label{eq:Sigma1}
\Sigma_1=\diag(\sigma_1,\sigma_2,\ldots,\sigma_r),\quad
\sigma_1\ge\cdots\ge\sigma_r>0.
\end{equation}
and, similarly for $\wtd\Sigma_1$.
Our main result of this paper is stated in the following theorem.

\begin{theorem} \label{thm:main}
Given $D\in\bbR^{n\times k}$ and $k$-dimensional subspaces $\cX$ and $\wtd\cX$ of $\bbR^n$,
let $X,\,\wtd X\in\bbO^{n\times k}$ such that both \eqref{eq:charz-X} and \eqref{eq:charz-tX} hold.
Suppose that $\rank(X^{\T}D)=\rank(\wtd X^{\T}D)=:r$.
Then for any unitarily invariant norm $\|\cdot\|_{\UI}$
\begin{equation} \label{ineq:main}
\min_{Y\in\bbX}\|\wtd X-Y\|_{\UI}\le\eta\,\|\sin\Theta(\cX,\wtd\cX)\|_{\UI},
\end{equation}
where $\bbX$ is defined as in~\eqref{eq:xset}, and
\begin{equation}\label{eq:main-eta}
\eta=
\begin{cases}
  \sqrt 2\left(1+\dfrac{2\,\|D\|_2}{\sigma_k+\tilde\sigma_k}\right),  &  \mbox{if $r=k$},\\[1em]
  \sqrt 2\left(1+\dfrac{2\,\|D\|_2}{\sigma_r+\tilde\sigma_r}\right)+\dfrac{(2\sqrt 2+4)\,\|D\|_2}{\max\{\sigma_r,\tilde\sigma_r\}},&\mbox{if $r<k$}.
\end{cases}
\end{equation}
Inequality \eqref{ineq:main} can be improved for  the matrix spectral and Frobenius norm,
in the case when $r<k$, with a smaller $\eta$ given by
\begin{equation}\label{eq:main-eta'}
\eta=
\begin{cases}
  \sqrt 2\left(1+\dfrac{2\,\|D\|_2}{\sigma_r+\tilde\sigma_r}\right)
         +\dfrac{4\,\|D\|_2}{\max\{\sigma_r,\tilde\sigma_r\}},
                    &  \mbox{for $\|\cdot\|_{\UI}=\|\cdot\|_{\F}$},\\[1em]
  \sqrt 2+\sqrt{\dfrac{8\,\|D\|_2^2}{(\sigma_r+\tilde\sigma_r)^2}
             +\dfrac{4\,\|D\|_2^2}{[\max\{\sigma_r,\tilde\sigma_r\}]^2}}
             +\dfrac{4\,\|D\|_2}{\max\{\sigma_r,\tilde\sigma_r\}},
                    &\mbox{for $\|\cdot\|_{\UI}=\|\cdot\|_2$}.
\end{cases}
\end{equation}
\end{theorem}

\begin{remark}\label{rk:main}
There are a few comments in order.
\begin{enumerate}[(a)]
\item For the case $r=k$, the left-hand side of~\eqref{ineq:main} is really $\|\wtd X-X\|_{\UI}$, yielding
      \begin{equation} \label{eq:main;r=k}
      \|\wtd X-X\|_{\UI}\le\eta\,\|\sin\Theta(\cX,\wtd\cX)\|_{\UI},
      \end{equation}
      because then $\bbX=\{X\}$
      as we previously explained.

\item The coefficient $\eta$ is smallest when $r=k$, for any general unitarily invariant norm and for the
      two specific ones: the matrix spectral and Frobenius norm.
      Both values for $\eta$ in \eqref{eq:main-eta'} are smaller than
      the ones in \eqref{eq:main-eta} for the case $r<k$. This is easily seen for the first value in \eqref{eq:main-eta'};
      for the second value, we may use
      $\sqrt{a^2+b^2}\le a+b$ for all $a,\,b\ge 0$ to see the fact.


\item Similarly to the definition of set $\bbX$, we may define, associated with $\wtd X$,
      \begin{equation}\label{eq:txset}
      \wtd\bbX=\left\{\wtd X\wtd V\begin{bmatrix}
                           I_r &  \\
                            & W
                         \end{bmatrix}\wtd V=\wtd X\wtd V_1\wtd V_1^{\T}+\wtd X\wtd V_2 W\wtd V_2^{\T}\,:\, W\in\bbO^{r'\times r'}\right\}.
      \end{equation}
      As we explained before, the term $\wtd X\wtd V_1\wtd V_1^{\T}$ depends on $\wtd\cX$ only, and
      $\wtd X$ is just one of the elements in $\wtd\bbX$ and can be any one in the set as far as the conclusion
      of Theorem~\ref{thm:main} is concerned. Hence \eqref{ineq:main} leads to a bound on
      the Hausdorff distance~\cite[Section~11.1]{pete:2006}
      between $\bbX$ and $\wtd\bbX$
      \begin{equation}\label{eq:haus-dist}
      \dist(\bbX,\wtd\bbX)=\max_{\wtd Y\in\wtd\bbX}\,\min_{Y\in\bbX}\|\wtd Y-Y\|_{\UI}\le\eta\,\|\sin\Theta(\cX,\wtd\cX)\|_{\UI},
      \end{equation}
      which can be interpreted as for any point $\wtd Y$ in $\wtd\bbX$
      there is a point $Y$ in $\bbX$ that is no further than $\eta\|\sin\Theta(\cX,\wtd\cX)\|_{\UI}$ away from $\wtd Y$.
\end{enumerate}
\end{remark}

The rest of this section is devoted to the proof of Theorem~\ref{thm:main}. For that purpose, we notice that,
by Lemma~\ref{le:max-angle}, there exists an orthogonal matrix $Q\in\bbO^{k\times k}$
such that $\what X=\wtd X Q^{\T}\in\bbO^{n\times k}$ satisfies
\begin{equation}\label{eq:(X,tX)}
    \|X-\what X\|_{\UI}\le\sqrt 2\,\|\sin\Theta(\cX,\wtd\cX)\|_{\UI}.
\end{equation}
Note $\cR(\what X)=\wtd\cX$ and also $\wtd X=\what XQ$. Using \eqref{eq:(X,tX)}, we get
\begin{equation}\label{eq:xd-x0td}
\|X^{\T}D-\what X^{\T}D\|_{\UI}\le\|X^{\T}-\what X^{\T}\|_{\UI}\|D\|_2
\le\sqrt 2\,\|D\|_2\,\|\sin\Theta(\cX,\wtd\cX)\|_{\UI}.
\end{equation}

\subsection{Case $\rank(X^{\T}D)=\rank(\wtd X^{\T}D)=k$}
In this case, $X^{\T}D\succ 0$ and $\wtd X^{\T}D\succ 0$, and both $X$ and $\wtd X$ are unique.
Therefore, we can bound $\|X-\wtd X\|_{\UI}$.

Observe that $\what X^{\T}D=Q(\wtd X^{\T}D)$ which is the polar decomposition of $\what X^{\T}D$ because $Q\in\bbO^{k\times k}$ and
$\wtd X^{\T}D\succeq 0$, while
$X^{\T}D=I_k\cdot(X^{\T}D)$ is the polar decomposition of $X^{\T}D$. Hence by Lemma~\ref{le:polar}, we have
$$
\|I_k-Q\|_{\UI}\le\frac{2}{\sigma_k+\tilde\sigma_k}\|X^{\T}D-\what X^{\T}D\|_{\UI}
     \le\frac{2\,\|D\|_2}{\sigma_k+\tilde\sigma_k}\|X^{\T}-\what X^{\T}\|_{\UI}.
$$
Finally, we have, using \eqref{eq:(X,tX)},
\begin{align}
\|X-\wtd X\|_{\UI}&=\|X-\what XQ\|_{\UI} \nonumber\\
 &=\|X-\what X+\what X-\what XQ\|_{\UI} \nonumber\\
 &\le\|X-\what X\|_{\UI}+\|\what X\|_2\|I_k-Q\|_{\UI} \nonumber\\
 &\le\|X-\what X\|_{\UI}
       +\frac{2\,\|D\|_2}{\sigma_k+\tilde\sigma_k}\|X^{\T}-\what X^{\T}\|_{\UI} \nonumber\\
&\le\left(1+\frac{2\,\|D\|_2}{\sigma_k+\tilde\sigma_k}\right)\sqrt 2\,\|\sin\Theta(\cX,\wtd\cX)\|_{\UI},
  \label{eq:upb1}
\end{align}
yielding \eqref{ineq:main} for the case.

\subsection{$\rank(X^{\T}D)=\rank(\wtd X^{\T}D)<k$}

In the current case, both $X$ and $\wtd X$ are not uniquely determined by $X^{\T}D\succeq 0$ and $\wtd X^{\T}D\succeq 0$.
Hence it only make sense to bound $\|\wtd X-Y\|_{\UI}$ subject to $Y\in\bbX$.

Recall $\what X=\wtd X Q^{\T}$ introduced to satisfy \eqref{eq:(X,tX)}. Evidently,
$\rank(\what X^{\T}D)=\rank(\wtd X^{\T}D)=\rank(X^{\T}D)=r<k$. The SVD of $\what X^{\T}D=Q(\wtd X^{\T}D)$ can be given as
\begin{equation}\label{eq:svd-hatXD}
\what X^{\T}D=(Q\wtd U)\wtd\Sigma\wtd V^{\T}\equiv\kbordermatrix{ &\sss r &\sss r'\\
                                 & Q\wtd U_1 & Q\wtd U_2}\times
                      \kbordermatrix{ &\sss r &\sss r' \\
                                    \sss r & \wtd\Sigma_1 & \\
                                    \sss r' & & 0}\times
                      \kbordermatrix{ &\\
                               \sss r  & \wtd V_1^{\T} \\
                               \sss r'  & \wtd V_2^{\T}},
\end{equation}
and let $\what U=Q\wtd U$, $\what U_1=Q\wtd U_1$, and $\what U_2=Q\wtd U_2$.
By Lemma~\ref{le:mp},
there exists a $\wtd W\in\bbO^{r'\times r'}$ such that
\begin{equation}\label{eq:def-tx-nofullrank}
\wtd X=\what X\what U\begin{bmatrix}
                     I_r &  \\
                      & \wtd W
                   \end{bmatrix}\wtd V^{\T}=\what X\what U_1\wtd V_1^{\T}+\what X\what U_2 \wtd W\wtd V_2^{\T}.
\end{equation}
We have, by~\eqref{eq:xset} and \eqref{eq:def-tx-nofullrank},
\begin{align}
\min_{Y\in\bbX}\|\wtd X-Y\|_{\UI}
    &=\min_{W\in\bbO^{r'\times r'}}
        \left\|\what X\what U\begin{bmatrix}
                     I_r &  \\
                      & \wtd W
                   \end{bmatrix}\wtd V^{\T}- XV\begin{bmatrix}
                     I_r &  \\
                      & W
                   \end{bmatrix}V^{\T}\right\|_{\UI} \notag\\
  &\le\min_{W\in\bbO^{r'\times r'}}\left(\left\|\what X\what U\begin{bmatrix}
                     I_r &  \\
                      & \wtd W
                   \end{bmatrix}\wtd V^{\T}- \what XV\begin{bmatrix}
                     I_r &  \\
                      & W
                   \end{bmatrix}V^{\T}\right\|_{\UI}\right. \notag \\
  &\qquad\qquad\qquad \left.+\left\|\what XV\begin{bmatrix}
                     I_r &  \\
                      & W
                   \end{bmatrix}V^{\T}- XV\begin{bmatrix}
                     I_r &  \\
                      & W
                   \end{bmatrix}V^{\T}\right\|_{\UI}\right)  \notag\\
  &=\|\what X- X\|_{\UI}+\min_{W\in\bbO^{r'\times r'}}\left\|\what X\what U\begin{bmatrix}
                     I_r &  \\
                      & \wtd W
                   \end{bmatrix}\wtd V^{\T}- \what XV\begin{bmatrix}
                     I_r &  \\
                      & W
                   \end{bmatrix}V^{\T}\right\|_{\UI}  \notag\\
  &\le\|\what X- X\|_{\UI}+\min_{W\in\bbO^{r'\times r'}}\|\what U_1\wtd V_1^{\T}-V_1V_1^{\T}
        +\what U_2 \wtd W\wtd V_2^{\T}-V_2WV_2^{\T}\|_{\UI}\notag \\
  &\le\|\what X- X\|_{\UI}+\|\what U_1\wtd V_1^{\T}-V_1V_1^{\T}\|_{\UI}
      +\min_{W\in\bbO^{r'\times r'}}\|\what U_2\wtd W\wtd V_2^{\T}- V_2WV_2^{\T}\|_{\UI}.
         \label{eq:dist-all}
\end{align}
Now, we shall bound the three terms in the right-hand side of \eqref{eq:dist-all} in terms of
$\|\sin\Theta(\cX,\wtd\cX)\|_{\UI}$.
First $\|\what X- X\|_{\UI}$ has already been taken care of by \eqref{eq:(X,tX)}.
Next, it can be seen that $\what U_1\wtd V_1^{\T}$ and $V_1V_1^{\T}$ are the canonical isometry polar factors of
$\what X^{\T}D$ and $X^{\T}D$, respectively.
By Lemma~\ref{le:polar} and \eqref{eq:xd-x0td}, we have
\begin{align}\label{eq:dist-1}
 \notag \|\what U_1\wtd V_1^{\T}-V_1V_1^{\T}\|_{\UI}&\le\left(\frac{2}{\sigma_r+\tilde\sigma_r}
 +\frac{2}{\max\{ \sigma_r,\tilde\sigma_r\}}\right)\,\|X^{\T}D-\what X^{\T}D\|_{\UI}\\
&\le \left(\frac{2}{\sigma_r+\tilde\sigma_r}+\frac{2}{\max\{\sigma_r,\tilde\sigma_r\}}\right)
  \sqrt 2\,\|D\|_2\,\|\sin\Theta(\cX,\wtd\cX)\|_{ \ui}.
\end{align}
Finally, let $\what\cU_i=\cR(\what U_i)$ and $\cV_i=\cR(V_i)$ for $i=1,2$.
It follows from~\eqref{eq:def-comp-sin} and Lemma~\ref{lm:SVD-pert:wedin} that
\begin{align}\label{eq:sinu2v2}
\notag \|\sin\Theta(\cV_2,\what\cU_2)\|_{\UI}&=\|V_1^{\T}\what U_2\|_{\UI}=\|\what U_2^{\T}V_1\|_{\UI}\\
\notag &=\|\sin\Theta(\cV_1,\what\cU_1)\|_{\UI}
         \le\frac{\|X^{\T}D-\what X^{\T}D\|_{\UI}}{\max\{ \sigma_r,\tilde\sigma_r\}} \\
&\le\frac{\sqrt 2\,\|D\|_2\,}{\max\{\sigma_r,\tilde\sigma_r\}}\|\sin\Theta(\cX,\wtd\cX)\|_{ \ui},
\end{align}
where the last inequality holds because of~\eqref{eq:xd-x0td}.
Note that $\what U_2\wtd W,\,V_2\in\bbO^{k\times r'}$ satisfying $\cR(\what U_2\wtd W)=\what\cU_2$ and $\cR(V_2)=\cV_2$.
Hence, by Lemma~\ref{le:max-angle},
there exists an orthogonal matrix $W_1\in\bbO^{r'\times r'}$ such that
\begin{equation} \label{eq:u2}
   \|V_2W_1-\what U_2\wtd W\|_{\UI}\le\sqrt 2\,\|\sin\Theta(\cV_2,\what\cU_2)\|_{ \ui}\le
   \frac{2\,\|D\|_2\,}{\max\{ \sigma_r,\tilde\sigma_r\}}\|\sin\Theta(\cX,\wtd\cX)\|_{ \ui}.
\end{equation}
Similarly, there exists $W_2\in\bbO^{r'\times r'}$ satisfying
\begin{equation} \label{eq:v2}
   \|V_2W_2-\wtd V_2\|_{\UI}\le\sqrt 2\,\|\sin\Theta(\cV_2,\wtd\cV_2)\|_{ \ui}\le
   \frac{2\,\|D\|_2\,}{\max\{ \sigma_r,\tilde\sigma_r\}}\|\sin\Theta(\cX,\wtd\cX)\|_{ \ui}.
\end{equation}
Keeping \eqref{eq:u2} and \eqref{eq:v2} in mind, we have
\begin{align}
\min_{W\in\bbO^{r'\times r'}}&\|\what U_2\wtd W\wtd V_2^{\T}- V_2WV_2^{\T}\|_{\UI} \notag \\
         &\le\|\what U_2\wtd W\wtd V_2^{\T}- V_2W_1W_2^{\T}V_2^{\T}\|_{\UI}  \notag\\
 &=\|\what U_2\wtd W\wtd V_2^{\T}- V_2W_1\wtd V_2^{\T}+V_2W_1\wtd V_2^{\T}-V_2W_1W_2^{\T}V_2^{\T}\|_{\UI} \notag \\
 &\le\|\what U_2\wtd W\wtd V_2^{\T}- V_2W_1\wtd V_2^{\T}\|_{\UI}
       +\|V_2W_1\wtd V_2^{\T}-V_2W_1W_2^{\T}V_2^{\T}\|_{\UI} \notag \\
 &\le\|\what U_2\wtd W- V_2W_1\|_{\UI}+\|\wtd V_2^{\T}-W_2^{\T} V_2^{\T}\|_{\UI} \nonumber\\
       &\le\frac{4\,\|D\|_2\,}{\max\{ \sigma_r,\tilde\sigma_r\}}\|\sin\Theta(\cX,\wtd\cX)\|_{\UI}.
           \label{eq:dist-2}
\end{align}

Together with~\eqref{eq:(X,tX)},~\eqref{eq:dist-all}, \eqref{eq:dist-1} and \eqref{eq:dist-2}, we have
\begin{align}
\min_{Y\in\bbX}\|\wtd X-Y\|_{\UI}
    &\le\left[\sqrt 2+\left(\frac{2}{\sigma_r+\tilde\sigma_r}
         +\frac{2}{\max\{\sigma_r,\tilde\sigma_r\}}\right)\sqrt 2\,\|D\|_2\right. \notag \\
    &\hspace{4cm}\left.     +\frac{4\,\|D\|_2}{\max\{ \sigma_r,\tilde\sigma_r\}}\right]\|\sin\Theta(\cX,\wtd\cX)\|_{\UI} \notag\\
    &\le\left(\sqrt 2+\frac{2\sqrt 2\,\|D\|_2}{\sigma_r+\tilde\sigma_r}
    +\frac{(2\sqrt 2+4)\,\|D\|_2}{\max\{\sigma_r,\tilde\sigma_r\}}\right)\|\sin\Theta(\cX,\wtd\cX)\|_{\UI}.
        \label{eq:upb-ui-nofull}
\end{align}
Inequalities \eqref{eq:upb1} and \eqref{eq:upb-ui-nofull} yield \eqref{ineq:main} for any general unitarily invariant norm.

Inequality \eqref{ineq:main} can be improved for two particular unitarily invariant norms, the matrix spectral and Frobenius norm.
In our case, the improvements are made possible by using better bounds than \eqref{eq:dist-1} when it comes to
the two particular norms, thanks to Lemma~\ref{le:polar}.

By Lemma~\ref{le:polar},
inequality~\eqref{eq:dist-1} can be improved, in the case of the Frobenius norm, to
\begin{align}\label{eq:dist-f1}
\notag \|\what U_1\wtd V_1^{\T}-V_1V_1^{\T}\|_{\F}&\le\frac{2}{\sigma_r+\tilde\sigma_r}\,
\|X^{\T}D-\what X^{\T}D\|_{\F}\\
&\le \frac{2\sqrt 2\,\|D\|_2}{\sigma_r+\tilde\sigma_r}\,\|\sin\Theta(\cX,\wtd\cX)\|_{\F}.
\end{align}
Therefore, together \eqref{eq:dist-all},~\eqref{eq:dist-2} and~\eqref{eq:dist-f1} lead to
\begin{equation}\label{eq:upb-fnorm-nofull}
\min_{Y\in\bbX}\|\wtd X-Y\|_{\F}\le\left(\sqrt 2+\frac{2\sqrt 2\,\|D\|_2}{\sigma_r+\tilde\sigma_r}
+\frac{4\,\|D\|_2}{\max\{\sigma_r,\tilde\sigma_r\}}\right)\,\|\sin\Theta(\cX,\wtd\cX)\|_{\F}.
\end{equation}
Similarly, when $\|\cdot\|_{\UI}=\|\cdot\|_2$, we have by~\eqref{eq:polar-2norm}
\begin{equation}\label{eq:upb-2norm-nofull}
\min_{Y\in\bbX}\|\wtd X-Y\|_2
    \le\left(\sqrt 2+\sqrt{\dfrac{8\,\|D\|_2^2}{(\sigma_r+\tilde\sigma_r)^2}
             +\dfrac{4\,\|D\|_2^2}{\max\{\sigma_r^2,\tilde\sigma_r^2\}}}
             +\frac{4\,\|D\|_2}{\max\{\sigma_r,\tilde\sigma_r\}}\right)
             \,\|\sin\Theta(\cX,\wtd\cX)\|_2.
\end{equation}
Inequalities \eqref{eq:upb-fnorm-nofull} and \eqref{eq:upb-2norm-nofull} yield
\eqref{ineq:main} with  improved $\eta$ given as in \eqref{eq:main-eta'}.

\begin{remark}\label{rk:main'}
Slight improvements to \eqref{ineq:main} for any general unitarily invariant norm are also possible from another direction.
Assuming $r=\rank(X^{\T}D)=\rank(\wtd X^{\T}D)<k$, we can have
\[
\max\big\{\|\sin\Theta(\cV_1,\what\cU_1)\|_{\UI},\; \|\sin\Theta(\cV_1,\wtd\cV_1)\|_{\UI}\big\}
\le\frac{\big\|\big(X^{\T}D-\what X^{\T}D\big)_{\bestr}\big\|_{\UI}}{\max\{\sigma_r,\tilde\sigma_r\}}
\]
by~\eqref{eq:sin-a},  and improve \eqref{eq:xd-x0td} to
\begin{align}
\big\|\big(X^{\T}D-\what X^{\T}D\big)_{\bestr}\big\|_{\UI}
    &\le \|D\|_2\,\big\|\big(X^{\T}-\what X^{\T}\big)_{\bestr}\big\|_{\UI} \nonumber \\
    &\le\sqrt 2\,\|D\|_2\,\big\|\big[\sin\Theta(\cX,\wtd\cX)\big]_{\bestr}\big\|_{\UI}.  \label{eq:xd-x0td'}
\end{align}
Note
$\big\|\big[\sin\Theta(\cX,\wtd\cX)\big]_{\bestr}\big\|_{\UI}
=\big\|\big(X_{\perp}^{\T}\wtd X\big)_{\bestr}\big\|_{\UI}$ in~\eqref{eq:sinu2v2} to obtain
\[
\|\sin\Theta(\cV_2,\what\cU_2)\|_{\UI}\le\frac{\sqrt 2\,\|D\|_2\,}{\max\{\sigma_r,\tilde\sigma_r\}}
\big\|\big[\sin\Theta(\cX,\wtd\cX)\big]_{\bestr}\big\|_{\UI}.
\]
Similarly, improvements to \eqref{eq:u2} and~\eqref{eq:v2} can be obtained as follows:
\begin{align*}
   \|V_2W_1-\what U_2\wtd W\|_{\UI}&\le\sqrt 2\,\|\sin\Theta(\cV_2,\what\cU_2)\|_{\UI}
\le\frac{2\,\|D\|_2\,}{\max\{\sigma_r,\tilde\sigma_r\}}
\big\|\big[\sin\Theta(\cX,\wtd\cX)\big]_{\bestr}\big\|_{\UI},\\
   \|V_2W_2-\wtd V_2\|_{\UI}&\le\sqrt 2\,\|\sin\Theta(\cV_2,\wtd\cV_2)\|_{\UI}
\le\frac{2\,\|D\|_2\,}{\max\{\sigma_r,\tilde\sigma_r\}}
\big\|\big[\sin\Theta(\cX,\wtd\cX)\big]_{\bestr}\big\|_{\UI},
\end{align*}
consequently a slightly sharper bound on $\min_{Y\in\bbX}\|\wtd X-Y\|_{\UI}$ than \eqref{ineq:main} by replacing
$\sin\Theta(\cX,\wtd\cX)$ there with $\big[\sin\Theta(\cX,\wtd\cX)\big]_{\bestr}$.
\end{remark}

\section{Numerical examples}\label{sec:num}
In this section, we conduct  numerical experiments
to demonstrate the effectiveness of the main result in this paper.

Let $M=\frac 1{\sqrt n}\,\hadm(n)$, where $\hadm$ is a  MATLAB function that generates a Hadamard matrix, which
is orthogonal. Let
$$
X_{\diamond}=M_{ (:,1:k)},\quad
\wtd X_{\diamond}=\sqrt{1-\delta^2}\,M_{(:,1:k)}Q_1+\delta M_{(:,k+1:2k)}Q_2,
$$
where $\delta$ is a parameter to control the distance between $\cX=\cR(X_{\diamond})$
and $\wtd\cX=\cR(\wtd X_{\diamond})$, $Q_1, Q_2\in\bbO^{k\times k}$ generated by
MATLAB's built-in functions $\orth$ and $\randn$ as $\orth(\randn(k))$. It can be calculated that
$X_{\diamond}^{\T}\wtd X_{\diamond}=\sqrt{1-\delta^2} Q_1$ whose singular values are
$\sqrt{1-\delta^2}$ of multiplicity $k$ and, hence,  the $k$ canonical angles $\theta_i$
between $\cX$ and $\wtd\cX$ are all the same with $\cos\theta_i=\sqrt{1-\delta^2}$, yielding
\begin{equation}\label{eq:sin-value}
\|\sin\Theta(\cX,\wtd\cX)\|_2=\delta,\,
\|\sin\Theta(\cX,\wtd\cX)\|_{\F}=\sqrt k\,\delta,\,
\|\sin\Theta(\cX,\wtd\cX)\|_{\tr}= k\,\delta.
\end{equation}
They all go to $0$ as $\delta$ does, but as orthonormal basis matrices of $\cX$ and $\wtd\cX$,
respectively, $X_{\diamond}$ and $\wtd X_{\diamond}$ are nowhere near.

Our main result in this paper shows that any $D\in\bbR^{n\times k}$ such that $X^{\T}D\succ 0,\,\wtd X^{\T}D\succ 0$
can nail down particular orthonormal basis matrices $X$ of $\cX$ and $\wtd X$  of $\wtd\cX$,
respectively, and that ensures $X-\wtd X=O(\delta)$. In what follows we will first numerically demonstrate
the sharpness of this upper bound for the
matrix norms $\|\cdot\|_{\UI}=\|\cdot\|_2,\, \|\cdot\|_{\F}$, and $\|\cdot\|_{\tr}$, as
$\delta\to 0$.
Specifically, let
$X=X_{\diamond}UV^{\T}$ and $\wtd X=\wtd X_{\diamond}\wtd U\wtd V^{\T}$
where $UV^{\T}$ and $\wtd U\wtd V^{\T}$ are the orthogonal polar factors of $X_{\diamond}^{\T}D$
and $\wtd X_{\diamond}^{\T}D$, respectively. Consider
$n=96$, $k=5$, and
\begin{equation}\label{eq:D-fullrank}
D=
\begin{bmatrix}
    1 & 0 & 0 & 0 & 0\\
    0 & 1 & 0 & 0 & 0\\
    0 & 0 & 1 & 0 & 0\\
    0 & 0 & 0 & 1 & 0\\
    0 & 0 & 0 & 0 & 1\\
    \frac 6 {8n} &  \frac 6 {8n+1} &  \frac 6 {8n+2} &  \frac 6 {8n+3} &  \frac 6 {8n+4} \\
    \vdots & \vdots & \vdots & \vdots &  \vdots\\
    \frac n {8n} &  \frac n {8n+1} &  \frac n {8n+2} &  \frac n {8n+3} &  \frac n {8n+4}
\end{bmatrix}\textcolor{red}{.}
\end{equation}
With this $D$, we have, by Theorem~\ref{ineq:main},
\begin{equation}\label{eq:egs:bd1}
\|X-\wtd X\|_{\UI}\le\xi_{\UI}:=\left(\sqrt 2+\frac{2\sqrt 2\,\|D\|_2}{\sigma_k+\tilde\sigma_k}\right)
                \times\|\sin\Theta(\cX,\wtd\cX)\|_{\UI}.
\end{equation}
In Figure~\ref{fig1}, we present three plots, each of which contains
the upper bound $\xi_{\UI}$ in \eqref{eq:egs:bd1} and the exact $\|X-\wtd X\|_{\UI}$ for the three norms, respectively.
It is observed that the upper bounds are tight and indicative of the true difference $\|X-\wtd X\|_{\UI}$.

\begin{figure}[t]
\begin{center}
\includegraphics[height=0.20\textheight,width=0.328\textwidth]{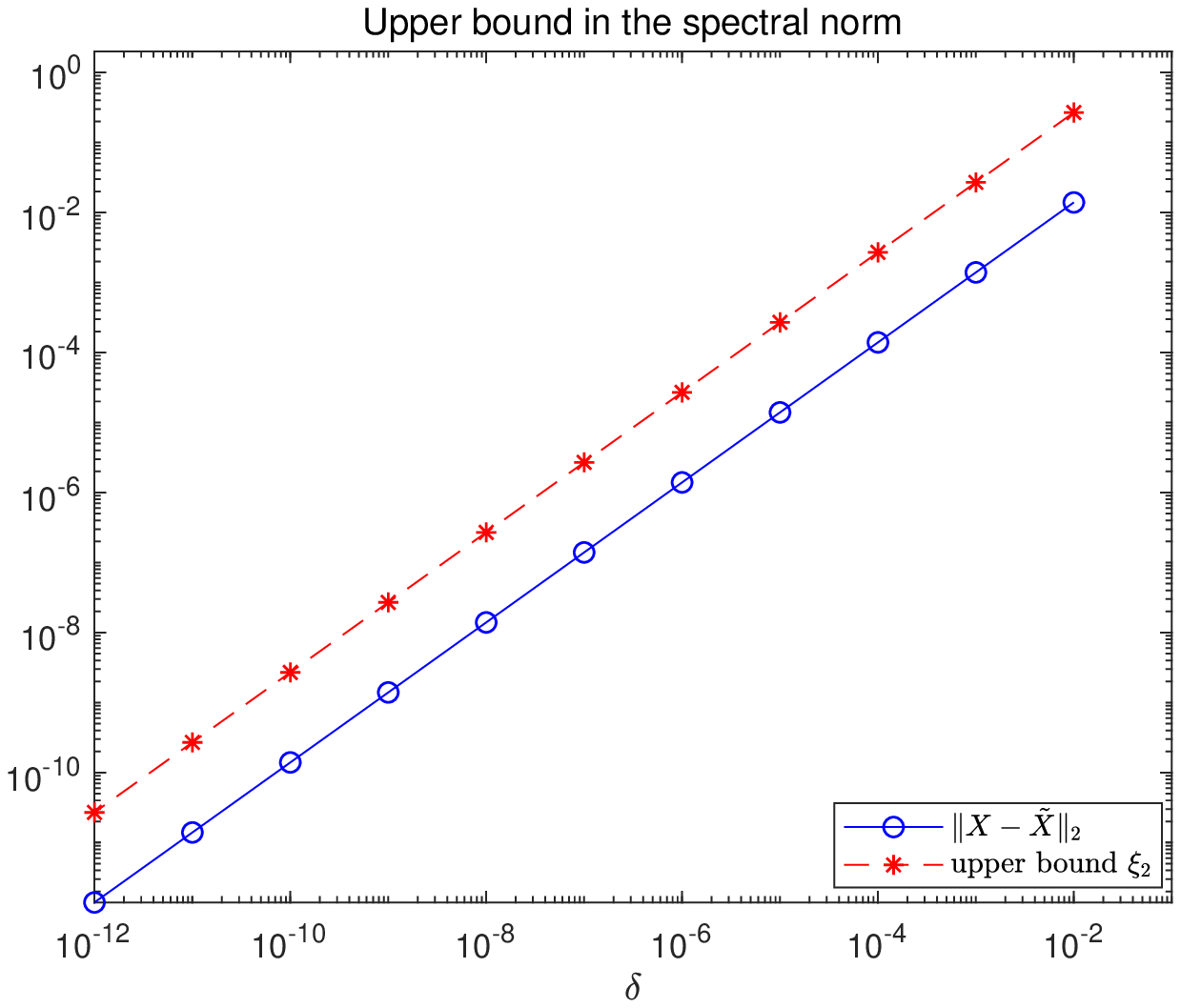}
\includegraphics[height=0.20\textheight,width=0.328\textwidth]{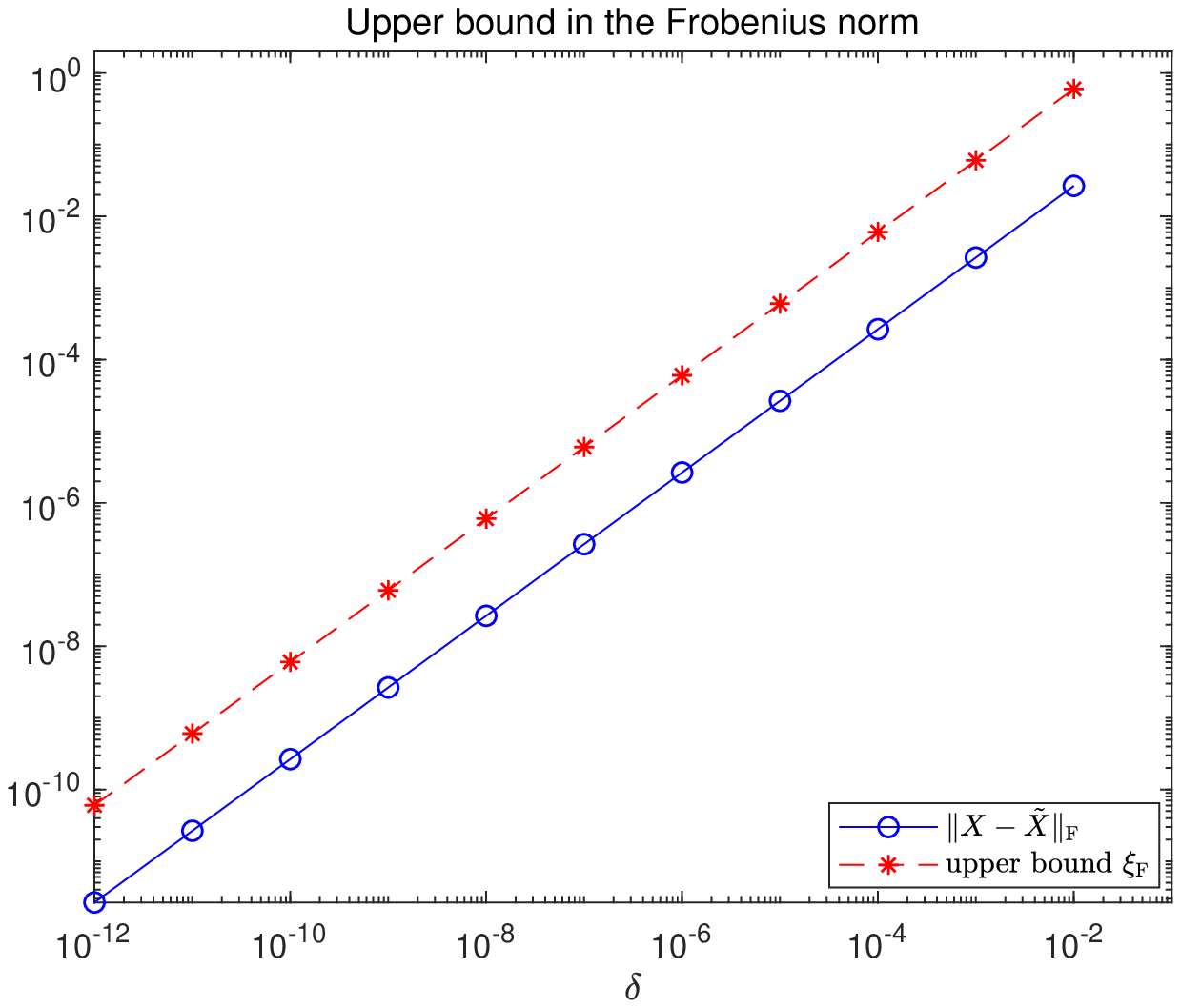}
\includegraphics[height=0.20\textheight,width=0.328\textwidth]{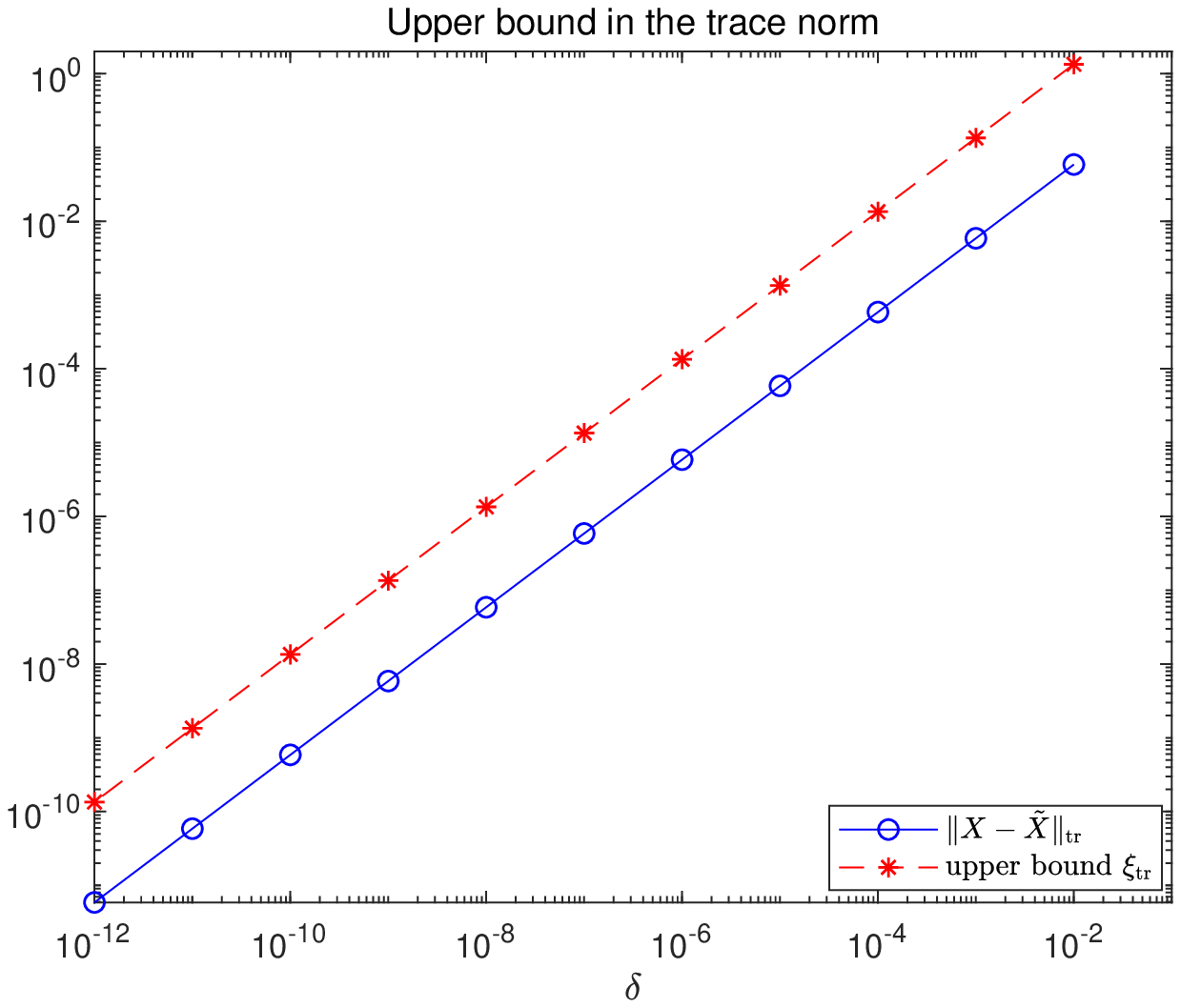}
\end{center}
\vspace{-10pt}
\caption{\small
         The full rank case: $r=k$ with $D$ as in \eqref{eq:D-fullrank}.
         Upper bound $\xi_{\UI}$ in \eqref{eq:egs:bd1} and the exact $\|X-\wtd X\|_{\UI}$
         as $\delta$ varies from $10^{-12}$ to $10^{-2}$.
         {\em Left:\/} the spectral norm; {\em Middle:\/} the Frobenius norm;
         {\em Right:\/} the trace norm.}
\label{fig1}
\end{figure}

Next, we consider the rank-deficient case: $\rank(X^{\T}D)=\rank(\wtd X^{\T}D)=:r<k$.
By Theorem~\ref{thm:main}, we have
\begin{subequations}\label{eq:def-xi-nofr}
\begin{equation}\label{eq:egs:bd2}
\min_{Y\in\bbX}\|\wtd X-Y\|_{\UI}\le\xi_{\UI},
\end{equation}
where for the
matrix norms $\|\cdot\|_{\UI}=\|\cdot\|_2,\, \|\cdot\|_{\F}$, and $\|\cdot\|_{\tr}$
\begin{align}
\xi_2&=\left(\sqrt 2+\sqrt{\dfrac{8\,\|D\|_2^2}{(\sigma_r+\tilde\sigma_r)^2}
       +\dfrac{4\,\|D\|_2^2}{\max\{\sigma_r^2,\tilde\sigma_r^2\}}}+\frac{4\,\|D\|_2}{\max\{\sigma_r,\tilde\sigma_r\}}\right)
         \,\|\sin\Theta(\cX,\wtd\cX)\|_2, \label{eq:def-xi-nofr:2}\\
\xi_{\F}&=\left(\sqrt 2+\frac{2\sqrt 2\,\|D\|_2}{\sigma_r+\tilde\sigma_r}+\frac{4\,\|D\|_2}
          {\max\{\sigma_r,\tilde\sigma_r\}}\right)\times\|\sin\Theta(\cX,\wtd\cX)\|_{\F}, \label{eq:def-xi-nofr:F}\\
\xi_{\tr}&=\left(\sqrt 2+\frac{2\sqrt 2\,\|D\|_2}{\sigma_r+\tilde\sigma_r}+\frac{(2\sqrt{2}+4)\,\|D\|_2}
          {\max\{\sigma_r,\tilde\sigma_r\}}\right)\times\|\sin\Theta(\cX,\wtd\cX)\|_{\tr}. \label{eq:def-xi-nofr:tr}
\end{align}
\end{subequations}

For $r=k-1$, we simply take the same $D$
in \eqref{eq:D-fullrank} but reset
its last column to $0$. By Lemma~\ref{le:mp}, there are only two $X$ that satisfy $X^{\T}D\succeq 0$ and
$\cR(X)=\cX:=\cR(X_{\diamond})$:
$$
X=X_{\diamond}U_{(:,1:k-1)}V_{(:,1:k-1)}^{\T}+X_{\diamond}U_{(:,k)}V_{(:,k)}^{\T}, \quad
X_-=X_{\diamond}U_{(:,1:k-1)}V_{(:,1:k-1)}^{\T}-X_{\diamond}U_{(:,k)}V_{(:,k)}^{\T}.
$$
The same can be said for $\wtd X$ that satisfies $\wtd X^{\T}D\succeq 0$ and $\cR(\wtd X)=\wtd\cX:=\cR(\wtd X_{\diamond})$.
Hence
\begin{equation}\label{eq:egs:r=k-1}
\min_{Y\in\bbX}\|\wtd X-Y\|_{\UI}=\min\{\|\wtd X-X\|_{\UI}, \|\wtd X-X_-\|_{\UI}\}.
\end{equation}
In Figure~\ref{fig2}, we again present three plots, each of which contains
the upper bound $\xi_{\UI}$ and the exact quantity in \eqref{eq:egs:r=k-1} for the three norms, respectively.
We observe similar behaviors to those in Figure~\ref{fig1} for the full-rank case.

\begin{figure}
\begin{center}
\includegraphics[height=0.20\textheight,width=0.328\textwidth]{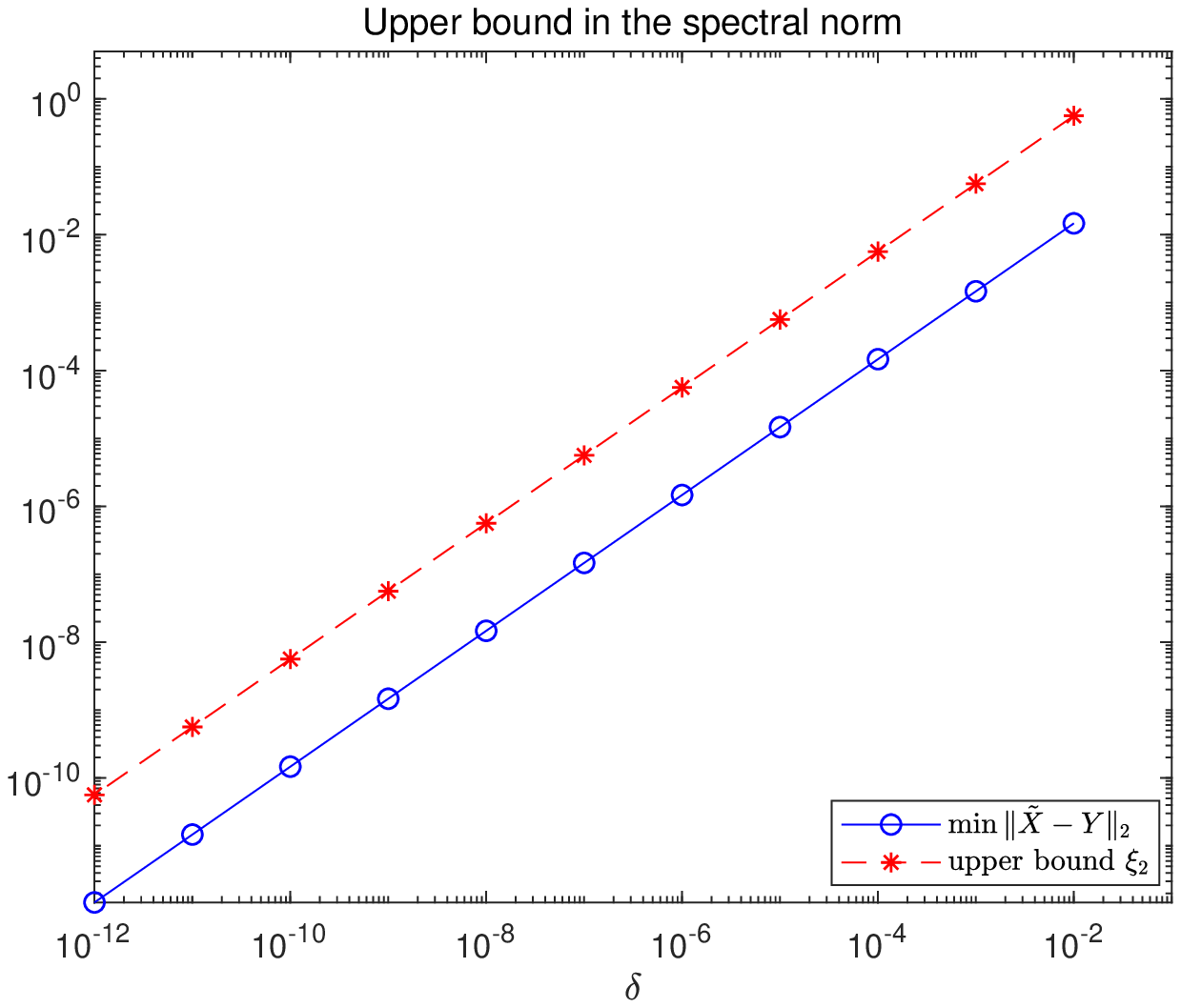}
\includegraphics[height=0.20\textheight,width=0.328\textwidth]{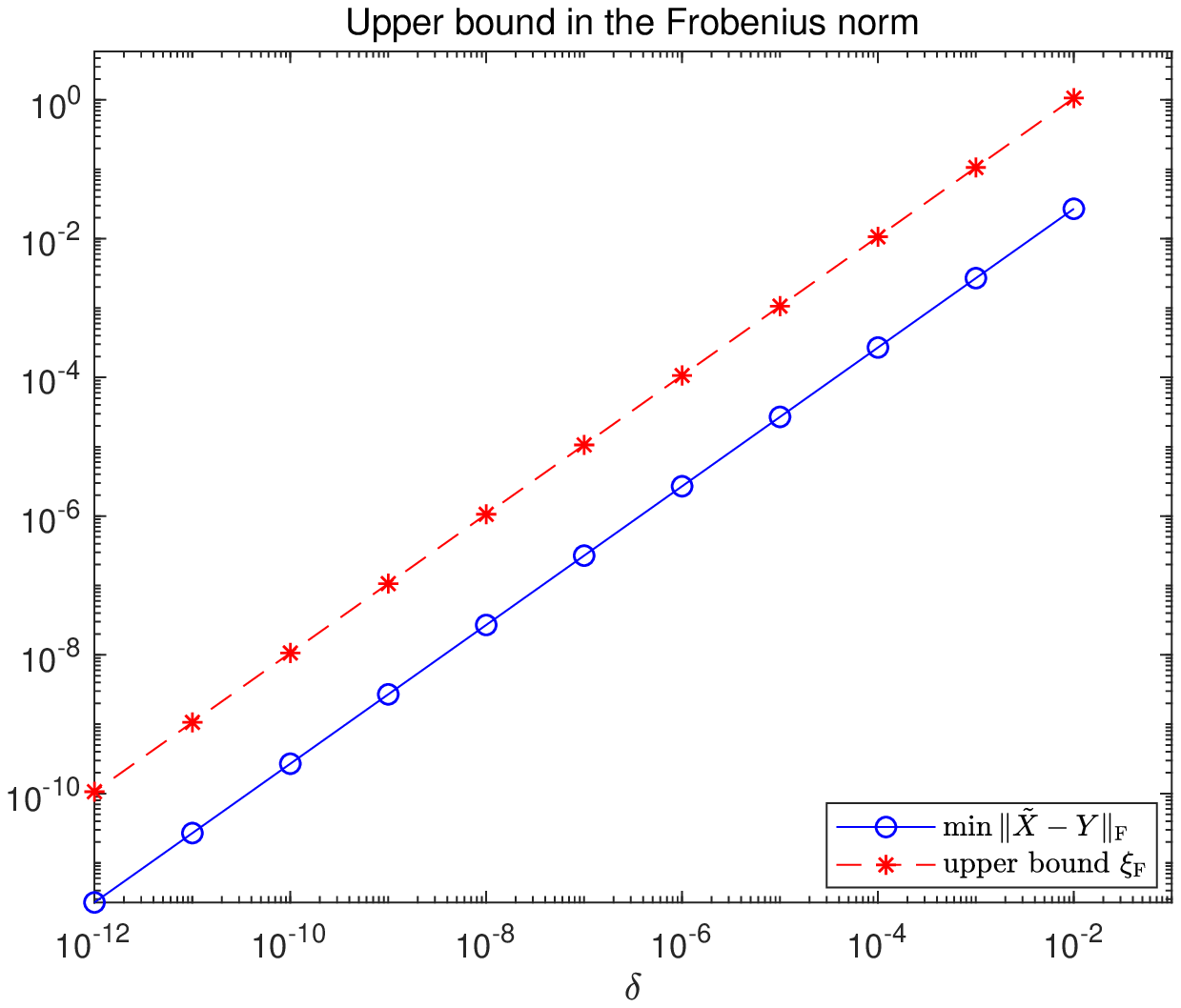}
\includegraphics[height=0.20\textheight,width=0.328\textwidth]{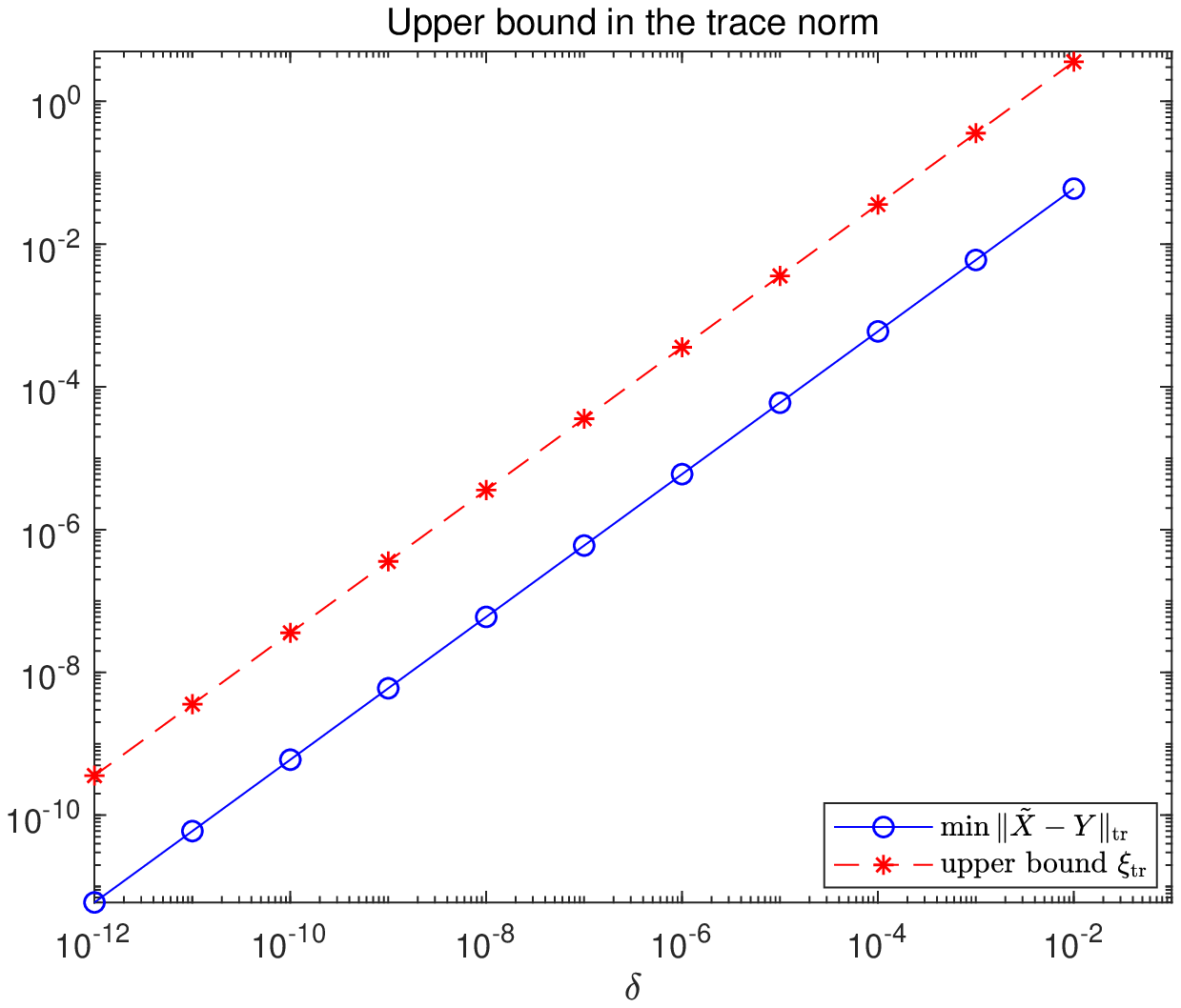}
\end{center}
\vspace{-10pt}
\caption{\small
         The rank-deficient case: $r=k-1$ by resetting the last column of $D$ in \eqref{eq:D-fullrank} to $0$.
         Upper bound $\xi_{\UI}$ in \eqref{eq:def-xi-nofr} and the exact $\min\|\wtd X-Y\|_{\UI}$ in
         \eqref{eq:egs:r=k-1}
         as $\delta$ varies from $10^{-12}$ to $10^{-2}$.
         {\em Left:\/} the spectral norm; {\em Middle:\/} the Frobenius norm;
         {\em Right:\/} the trace norm.}
\label{fig2}
\end{figure}

When $r<k-1$, the set $\bbX=\left\{X_{\sss\diamond}U_1V_1^{\T}+X_{\sss\diamond}U_2WV_2^{\T}\,:\, W\in\bbO^{(k-r)\times(k-r)}\right\}$
contains infinitely many elements. Fortunately,  for the Frobenius norm, we have
\begin{align*}
    \min_{Y\in\bbX}\|\wtd X-Y\|_{\F}^2&=\min_{W\in\bbO^{(k-r)\times(k-r)}}\left\|\wtd X-X_{\sss\diamond}U
                                         \begin{bmatrix}
                                            I_r &  \\
                                                &  W
                                            \end{bmatrix}V^{\T}\right\|_{\F}^2\\
                                       &=\min_{W\in\bbO^{(k-r)\times(k-r)}}\left\|\wtd XV-X_{\sss\diamond}U
                                         \begin{bmatrix}
                                            I_r &  \\
                                                &  W
                                         \end{bmatrix}\right\|_{\F}^2\\
                                       &=\min_{W\in\bbO^{(k-r)\times(k-r)}}\left\|
                                       \begin{bmatrix}
                                           \wtd XV_1, &\wtd XV_2
                                       \end{bmatrix}-\begin{bmatrix}
                                                        X_{\sss\diamond}U_1, & X_{\sss\diamond}U_2W
                                                      \end{bmatrix}
                                        \right\|_{\F}^2\\
                                       &=\|\wtd XV_1-X_{\sss\diamond}U_1\|_{\F}^2+\min_{W\in\bbO^{(k-r)\times(k-r)}}
                                       \|\wtd XV_2-X_{\sss\diamond}U_2W\|_{\F}^2,
\end{align*}
where the last term can be expressed as
\begin{equation*}
    \min_{W\in\bbO^{(k-r)\times(k-r)}}\|\wtd XV_2-X_{\sss\diamond}U_2W\|_{\F}^2=2(k-r)^2-
    \max_{W\in\bbO^{(k-r)\times(k-r)}}\tr(W^{\T}U_2^{\T}X_{\sss\diamond}^{\T}\wtd XV_2).
\end{equation*}
By von Neumann's trace inequality~\cite[section II.3.1]{stsu:1990}, the optimizer $W_{\sss\opt}$ for
\[
\max_{W\in\bbO^{(k-r)\times(k-r)}}\tr(W^{\T}U_2^{\T}X_{\sss\diamond}^{\T}\wtd XV_2)
\]
is the orthogonal polar factor of $U_2^{\T}X_{\sss\diamond}^{\T}\wtd XV_2$, and finally
\begin{equation}\label{eq:opt-form}
\min_{Y\in\bbX}\|\wtd X-Y\|_{\F}=\|\wtd X-Y_{\opt}\|_{\F},
\end{equation}
where $Y_{\opt}=X_{\sss\diamond}U_1V_1^{\T}+X_{\sss\diamond}U_2W_{\opt}V_2^{\T}$.
For norms other than $\|\cdot\|_{\F}$, it is not easy to calculate $\min\|\wtd X-Y\|_{\UI}$ subject to $Y\in\bbX$. For those norms,
we can tightly bound the exact $\min_Y\|\wtd X-Y\|_{\UI}$ tightly with the help of this $Y_{\opt}$ as follows:
\begin{subequations} \label{eq:bound-by-eql}
    \begin{align}
\frac 1{\sqrt k}\|\wtd X-Y_{\opt}\|_{\F}\le&\min_{Y\in\bbX}\|\wtd X-Y\|_2\le\|\wtd X-Y_{\opt}\|_2, 
        \label{eq:bound-by-eql-2} \\
\|\wtd X-Y_{\opt}\|_{\F}\le&\min_{Y\in\bbX}\|\wtd X-Y\|_{\tr}\le\|\wtd X-Y_{\opt}\|_{\tr}, 
        \label{eq:bound-by-eql-ncl}
    \end{align}
\end{subequations}
because $Y_{\opt}\in\bbX$ and for any $Y\in\bbX$,
$$
\frac 1{\sqrt k}\|\wtd X-Y\|_{\F}\le\|\wtd X-Y\|_2, \quad
\|\wtd X-Y\|_{\F}\le\|\wtd X-Y\|_{\tr}.
$$
We plot in Figure~\ref{fig13} upper bounds $\xi_{\UI}$ for $\min_Y\|\wtd X-Y\|_{\UI}$ the three norms,
their lower and upper bounds in \eqref{eq:bound-by-eql}.
It is noted  that
our upper bounds $\xi_{\UI}$ are again very good and go to $0$ at the same rates as the true ones.
\begin{figure}
\begin{center}
\includegraphics[height=0.20\textheight,width=0.328\textwidth]{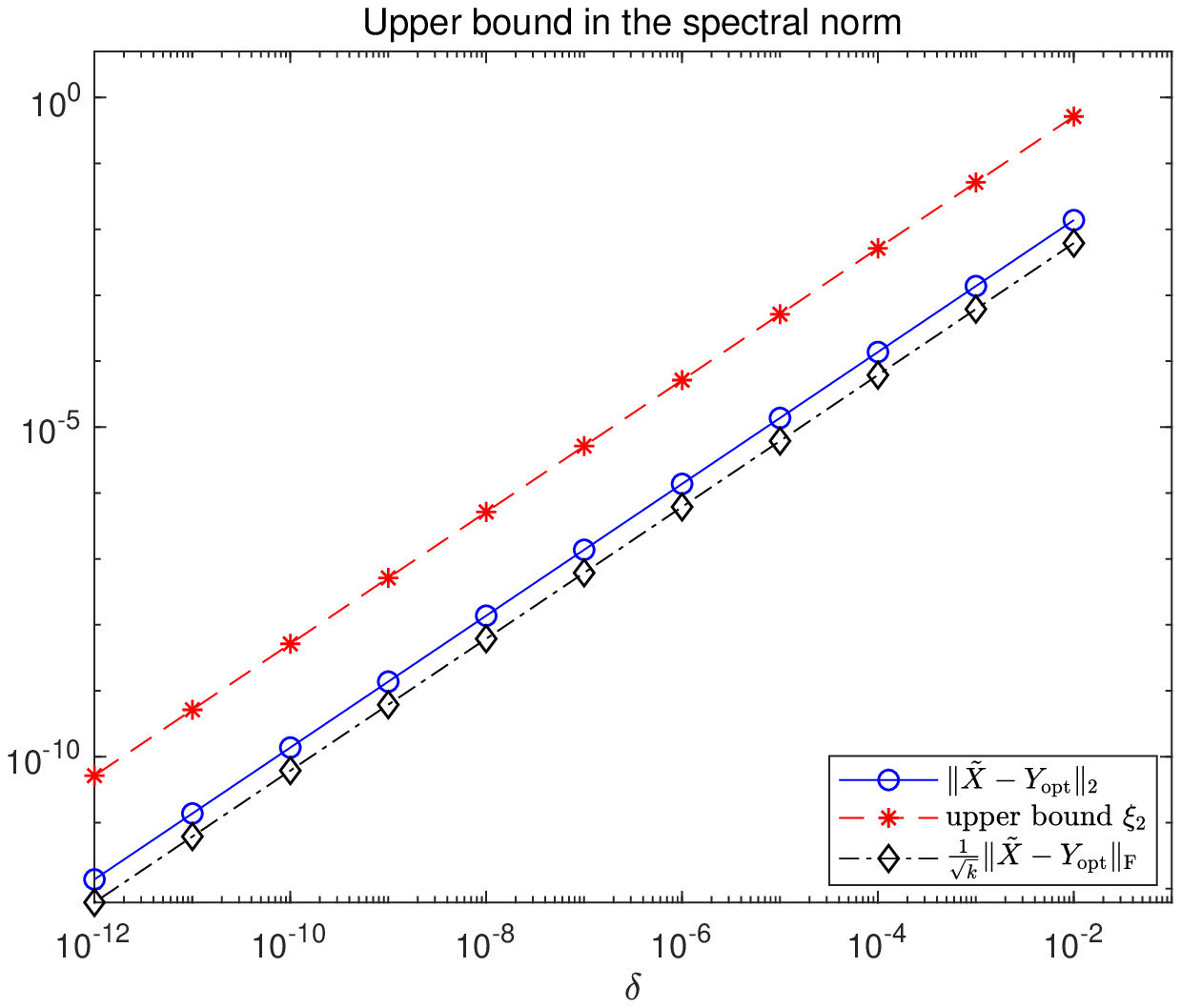}
\includegraphics[height=0.20\textheight,width=0.328\textwidth]{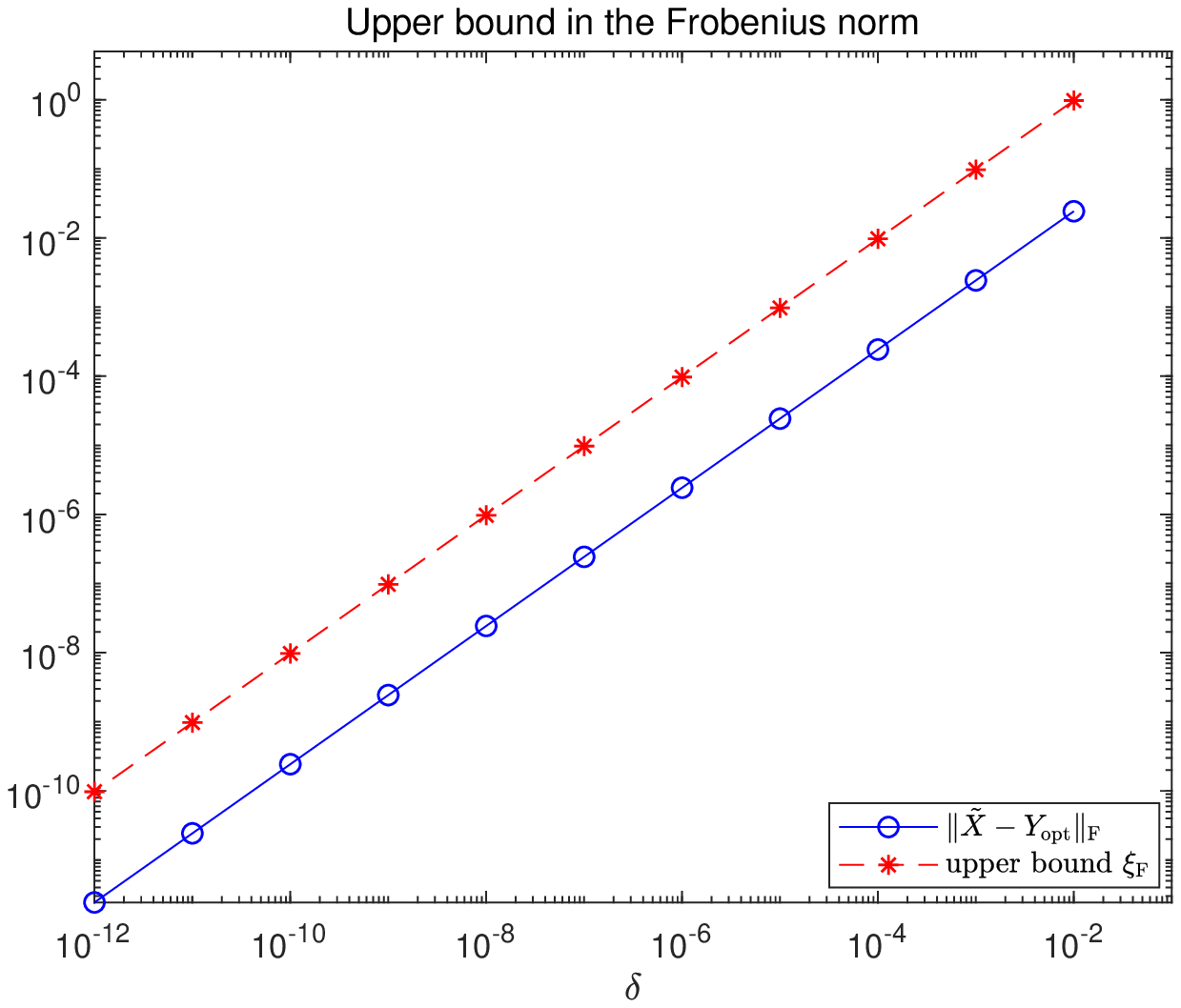}
\includegraphics[height=0.20\textheight,width=0.328\textwidth]{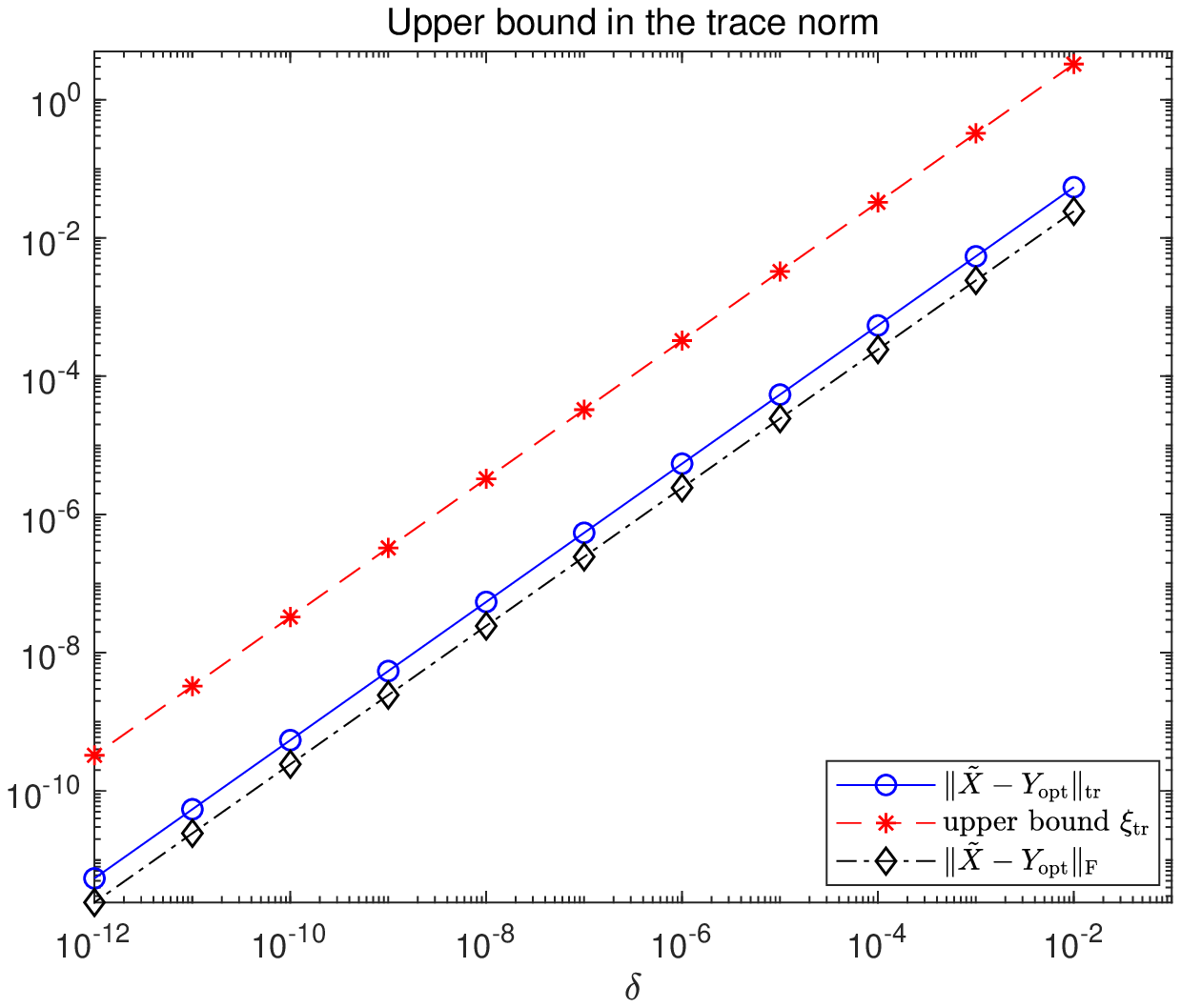}
\end{center}
\vspace{-10pt}
\caption{The rank-deficient case: $r=k-2$ by resetting the last two column of $D$ in \eqref{eq:D-fullrank} to $0$.
$\|\wtd X-Y_{\opt}\|_{\UI}$,  $\xi_{\UI}$ in \eqref{eq:def-xi-nofr} and bounds by \eqref{eq:bound-by-eql}.
         {\em Left:\/} the spectral norm; {\em Middle:\/} the Frobenius norm;
         {\em Right:\/} the trace norm.
}
\label{fig13}
\end{figure}

\section{Concluding remarks}\label{sec:col}
Let $\cX$ be a $k$-dimensional subspace of $\bbR^n$, and $D\in\bbR^{n\times k}$,
and let $X$ be an orthonormal basis matrix of $\cX$.
If $X^{\T}D\succ 0$, then $X$ is unique among all
orthonormal basis matrices of $\cX$. However if $X^{\T}D\succeq 0$ and $r=\rank(X^{\T}D)<k$, then
$X$ can be any one from set $\bbX$ in \eqref{eq:xset} of orthonormal basis matrices of $\cX$.
These results are recently obtained in \cite{luli:2022,wazl:2022}. In this paper, however,
we study how quantitatively $X$ changes as $\cX$ changes in both cases.

Specifically, suppose that $\cX$ is changed to $\wtd\cX$ and their difference is measured by
$\sin\Theta(\cX,\wtd\cX)$, the
sines of their canonical angles, and let $\wtd X$ be an orthonormal basis matrix of $\wtd\cX$.
In the case when both $X^{\T}D\succ 0,\,\wtd X^{\T}D\succ 0$, we
established upper bounds on $\|X-\wtd X\|$ in terms of $\|\sin\Theta(\cX,\wtd\cX)\|$ for the spectral, Frobenius
and, more generally, any unitarily invariant norm,
while in the case when both $X^{\T}D\succeq 0,\,\wtd X^{\T}D\succeq 0$ and also $\rank(X^{\T}D)=\rank(\wtd X^{\T}D)<k$,
our bounds are essentially on the Hausdorff distances of two sets $\bbX$ and $\wtd\bbX$ (see \eqref{eq:haus-dist}).
Numerical tests are conducted to  demonstrate the sharpness of our bounds.

The result in this paper can be used to understand the convergence of the SCF iteration in the NEPv approach
to solve maximization problems over the Stiefel manifold whose objective functions contain and increase with $\tr(X^{\T}D)$
\cite{luli:2022,wazl:2022,zhys:2020}
and even to assess approximation accuracy during SCF iterations.

Although our analysis so far focuses on the real number field, all developments can be
extended to the complex number field straightforwardly. To that end, only a few minor modifications are needed, namely,
replacing all $\bbR$ by $\bbC$, all matrix/vector transposes by complex conjugate transposes and $\tr(\,\cdot\,)$ by
$\re\left(\tr(\,\cdot\,)\right)$, where $\re(\,\cdot\,)$ takes the real part of a complex number.


{\small

\begin{thebibliography}{10}

\bibitem{ball:2022}
Z. Bai, R.-C. Li, and D. Lu.
\newblock Sharp estimation of convergence rate for self-consistent field
  iteration to solve eigenvector-dependent nonlinear eigenvalue problems.
\newblock {\em SIAM J. Matrix Anal. Appl.}, 43(1):301--327, 2022.

\bibitem{begr:2003}
A. {Ben-Israel} and T. N.~E. Greville.
\newblock {\em Generalized Inverses: Theory and Applications}.
\newblock Springer, New York, 2nd edition, 2003.

\bibitem{cazb:2018}
Y.~Cai, L.-H. Zhang, Z.~Bai, and R.-C. Li.
\newblock On an eigenvector-dependent nonlinear eigenvalue problem.
\newblock {\em SIAM J. Matrix Anal. Appl.}, 39(3):1360--1382, 2018.

\bibitem{chtr:2001}
M.~T. Chu and N.~T. Trendafilov.
\newblock The orthogonally constrained regression revisited.
\newblock {\em J. Comput. Graph. Stat.}, 10(4):746--771, 2001.

\bibitem{cugh:2015}
J.~P. Cunningham and Z.~Ghahramani.
\newblock Linear dimensionality reduction: Survey, insights, and
  generalizations.
\newblock {\em J. Mach. Learning Res.}, 16:2859--2900, 2015.

\bibitem{edas:1999}
A.~Edelman, T.~A. Arias, and S.~T. Smith.
\newblock The geometry of algorithms with orthogonality constraints.
\newblock {\em SIAM J. Matrix Anal. Appl.}, 20(2):303--353, 1999.

\bibitem{elpa:1999}
L.~Eld\'en and H.~Park.
\newblock A {P}rocrustes problem on the {S}tiefel manifold.
\newblock {\em Numer. Math.}, 82:599--619, 1999.

\bibitem{govl:2013}
G.~H. Golub and C.~F. {Van Loan}.
\newblock {\em Matrix Computations}.
\newblock Johns Hopkins University Press, Baltimore, Maryland, 4th edition,
  2013.

\bibitem{godi:2004}
J.~C. Gower and G.~B. Dijksterhuis.
\newblock {\em Procrustes Problems}.
\newblock Oxford University Press, New York, 2004.

\bibitem{high:2008}
N.~J. Higham.
\newblock {\em Functions of Matrices: {Theory} and Computation}.
\newblock Society for Industrial and Applied Mathematics, Philadelphia, PA,
  USA, 2008.

\bibitem{hoko:64}
P.~Hohenberg and W.~Kohn.
\newblock Inhomogeneous electron gas.
\newblock {\em Phys. Rev.}, 136:B864--B871, 1964.

\bibitem{hojo:2013}
R.~A. Horn and C.~R. Johnson.
\newblock {\em Matrix Analysis}.
\newblock Cambridge University Press, New York, NY, 2nd edition, 2013.

\bibitem{huca:1962}
J.~R. Hurley and R.~B. Cattell.
\newblock The {P}rocrustes program: producing direct rotation to test a
  hypothesized factor structure.
\newblock {\em Comput. Behav. Sci.}, 7:258--262, 1962.

\bibitem{kosh:1965}
W.~Kohn and L.~J. Sham.
\newblock Self-consistent equations including exchange and correlation effects.
\newblock {\em Phys. Rev.}, 140:A1133--A1138, 1965.

\bibitem{li:1993b}
R.-C. Li.
\newblock A perturbation bound for the generalized polar decomposition.
\newblock {\em BIT}, 33:304--308, 1993.

\bibitem{li:1995}
R.-C. Li.
\newblock New perturbation bounds for the unitary polar factor.
\newblock {\em SIAM J. Matrix Anal. Appl.}, 16:327--332, 1995.

\bibitem{li:2014HLA}
R.-C. Li.
\newblock Matrix perturbation theory.
\newblock In L.~Hogben, R.~Brualdi, and G.~W. Stewart, editors, {\em Handbook
  of Linear Algebra}, page Chapter 21. CRC Press, Boca Raton, FL, 2nd edition,
  2014.

\bibitem{lisu:2003}
W.~Li and W.~Sun.
\newblock New perturbation bounds for unitary polar factors.
\newblock {\em SIAM J. Matrix Anal. Appl.}, 25(2):362--372, 2003.

\bibitem{luli:2022}
D. Lu and R.-C. Li.
\newblock Convergence of {SCF} for locally unitarily invariantizable {NEPv},
  2022.
\newblock DOI: {\tt 10.48550/ARXIV.2212.14098}.


\bibitem{ngbs:2010}
T.~Ngo, M.~Bellalij, and Y.~Saad.
\newblock The trace ratio optimization problem for dimensionality reduction.
\newblock {\em SIAM J. Matrix Anal. Appl.}, 31(5):2950--2971, 2010.

\bibitem{parl:1998}
B.~N. Parlett.
\newblock {\em The Symmetric Eigenvalue Problem}.
\newblock SIAM, Philadelphia, 1998.

\bibitem{pete:2006}
P.~Petersen.
\newblock {\em Riemannian Geometry}.
\newblock Springer, 2006.

\bibitem{robl:2012}
D.~Rocca, Z.~Bai, R.-C. Li, and G.~Galli.
\newblock A block variational procedure for the iterative diagonalization of
  non-{Hermitian} random-phase approximation matrices.
\newblock {\em J. Chem. Phys.}, 136:034111, 2012.

\bibitem{stsu:1990}
G.~W. Stewart and J.-G. Sun.
\newblock {\em Matrix Perturbation Theory}.
\newblock Academic Press, Boston, 1990.

\bibitem{sun:1987}
J.-G. Sun.
\newblock {\em Matrix Perturbation Analysis}.
\newblock Academic Press, Beijing, 1987.
\newblock In Chinese.

\bibitem{wazl:2022a}
L.~Wang, L.-H. Zhang, and R.-C. Li.
\newblock Maximizing sum of coupled traces with applications.
\newblock {\em Numer. Math.}, 152:587--629, 2022.
\newblock {\tt doi.org/10.1007/s00211-022-01322-y}.

\bibitem{wazl:2022}
L.~Wang, L.-H. Zhang, and R.-C. Li.
\newblock Trace ratio optimization with an application to multi-view learning.
\newblock {\em Math. Program.}, 2022.
\newblock {\tt doi.org/10.1007/s10107-022-01900-w}.

\bibitem{wedi:1972}
P.-{\AA}. Wedin.
\newblock Perturbation bounds in connection with singular value decomposition.
\newblock {\em BIT}, 12:99--111, 1972.

\bibitem{yaml:2009}
C.~Yang, J.~C. Meza, B.~Lee, and L.-W. Wang.
\newblock {KSSOLV}---a {MATLAB} toolbox for solving the {Kohn-Sham} equations.
\newblock {\em ACM Trans. Math. Software}, 36(2):1--35, 2009.

\bibitem{zhln:2010}
L.-H. Zhang, L.-Z. Liao, and M.~K. Ng.
\newblock Fast algorithms for the generalized {Foley-Sammon} discriminant
  analysis.
\newblock {\em SIAM J. Matrix Anal. Appl.}, 31(4):1584--1605, 2010.

\bibitem{zhln:2013}
L.-H. Zhang, L.-Z. Liao, and M.~K. Ng.
\newblock Superlinear convergence of a general algorithm for the generalized
  {Foley-Sammon} discriminant analysis.
\newblock {\em J. Optim. Theory Appl.}, 157(3):853--865, 2013.

\bibitem{zhys:2020}
L.-H. Zhang, W.~H. Yang, C.~Shen, and J.~Ying.
\newblock An eigenvalue-based method for the unbalanced {P}rocrustes problem.
\newblock {\em SIAM J. Matrix Anal. Appl.}, 41(3):957--983, 2020.

\bibitem{zhli:2014b}
L.-H. Zhang and R.-C. Li.
\newblock Maximization of the sum of the trace ratio on the {Stiefel} manifold,
  {II}: Computation.
\newblock {\em SCIENCE CHINA Math.}, 58(7):1549--1566, 2015.

\bibitem{zhwb:2022}
L.-H. Zhang, L.~Wang, Z. Bai, and R.-C. Li.
\newblock A self-consistent-field iteration for orthogonal canonical
  correlation analysis.
\newblock {\em IEEE Trans. Pattern Anal. Mach. Intell.}, 44(2):890--904, 2022.

\bibitem{zhdu:2006}
Z. Zhang and K. Du.
\newblock Successive projection method for solving the unbalanced {P}rocrustes problem.
\newblock {\em SCIENCE CHINA Math.}, 49(7):971--986, 2006.

\end{thebibliography}
\def\noopsort#1{}\def\l{\char32l}\def\v#1{{\accent20 #1}}
  \let\^^_=\v\def\hbk{hardback}\def\pbk{paperback}

}

\end{document}